\documentclass[twoside, 11pt]{article}
\usepackage{mathrsfs}
\usepackage{amssymb, amsmath, mathrsfs, amsthm}%, wasysym}
\usepackage{graphicx}
\usepackage{amsmath}
\usepackage{color}
\usepackage[top=2cm, bottom=2cm, left=2cm, right=2cm]{geometry}
\usepackage{float, caption}%, subcaption

\input{epsf}

\newcommand{\bd}{\begin{description}}
\newcommand{\ed}{\end{description}}
\newcommand{\bi}{\begin{itemize}}
\newcommand{\ei}{\end{itemize}}
\newcommand{\be}{\begin{enumerate}}
\newcommand{\ee}{\end{enumerate}}
\newcommand{\beq}{\begin{equation}}
\newcommand{\eeq}{\end{equation}}
\newcommand{\beqs}{\begin{eqnarray*}}
\newcommand{\eeqs}{\end{eqnarray*}}

\numberwithin{equation}{section}

\definecolor{DarkGreen}{rgb}{0.2, 0.6, 0.3}

\newcommand{\BB}{\mathcal{B}_2}
\newcommand{\Bn}{\mathcal{B}_n}
\usepackage{amsthm}

\newtheorem{theorem}{Theorem}[section]

\newtheorem{lemma}{Lemma}[section]
\newtheorem{definition}{Definition}
\newtheorem{corollary}[theorem]{Corollary}
\newtheorem{case}{Case}
\newtheorem{ccase}{Case}
\newtheorem{subcase}{Subcase}[case]

\newtheorem{claim}{Claim}[section]
\newtheorem{remark}{Remark}[section]

\newtheorem{proposition}[theorem]{Proposition}

\newtheorem{question}{Question}
\newtheorem{observation}{Observation}[section]
\begin{document}
\title{\textbf{Boolean lattice without small rainbow subposets} \footnote{
The work was supported by the Hungarian National
Research, Development and Innovation Office
NKFIH (No.~SSN135643 and K132696); the National Science Foundation of China
(Nos. 12471329, ~and~12061059); JSPS KAKENHI (No.~22K19773, and 23K03195)}}

\author{Gyula O.H. Katona \footnote{Alfr\'{e}d R\'{e}nyi Institute of Mathematics, Hungarian Academy of Sciences, Budapest Reaaltanoda utca 13-15, 1053, Hungary. {\tt
katona.gyula.oh@renyi.hu}}, \  Yaping Mao\footnote{Faculty of Environment and Information Sciences, Yokohama
National University, 79-2 Tokiwadai, Hodogaya-ku, Yokohama 240-8501, 
Japan. {\tt maoyaping@ymail.com; ozeki-kenta-xr@ynu.ac.jp}},  \  Kenta Ozeki
\footnotemark[3], \  Zhao
Wang\footnote{College of Science, China Jiliang University, Hangzhou
310018, China. {\tt
wangzhao@mail.bnu.edu.cn}},  \  Gang Yang
\footnote{Graduate School of Environment and Information Sciences, Yokohama
National University, 79-2 Tokiwadai, Hodogaya-ku, Yokohama 240-8501, 
Japan. {\tt
yang-gang-hn@ynu.jp}}}
\date{}
\maketitle

\begin{abstract}
A Boolean lattice $\mathcal{B}_n=(2^X, \leq)$ is the
power set of an $n$-element ground set $X$ equipped with inclusion relation.
For two posets $\mathcal{P}$ and $\mathcal{Q}$, we say that $\mathcal{Q}$ 
contains an \emph{induced copy} of $\mathcal{P}$ if there exists an injection 
$f : \mathcal{P} \to \mathcal{Q}$ such that $f(X) \le f(Y)$ if and only if 
$X \le Y$ in $\mathcal{P}$.
A $k$-coloring is exact if all colors are used at least once. For posets $\mathcal{Q}$ and $\mathcal{P}$, the \emph{Boolean Gallai-Ramsey
number} $\operatorname{GR}_{k}(\mathcal{Q}:\mathcal{P})$ is defined as the smallest $n$ such that any 
exact $k$-coloring of the sets in $\mathcal{B}_n$ contains either a rainbow induced copy of $\mathcal{Q}$ or
a monochromatic induced copy of $\mathcal{P}$ and the \emph{Boolean rainbow Ramsey
number} $\operatorname{RR}(\mathcal{Q}:\mathcal{P})$ is defined as the smallest $n$ such that any coloring of the sets in $\mathcal{B}_n$ contains either a rainbow induced copy of $\mathcal{Q}$ or
a monochromatic induced copy of $\mathcal{P}$. 

In this paper, we first study the structural properties of exact $k$-colorings of the sets in Boolean lattice without rainbow induced copy of small posets.
As the application of these results, we give exact values and some bounds of Boolean Gallai-Ramsey
numbers and Boolean rainbow Ramsey numbers, which improve a result of Chen, Cheng, Li, and Liu in 2020 and give an answer of a question proposed by Chang, Gerbner, Li, Methuku, Nagy, 
Patk\'{o}s, and Vizer in 2022. 
\\[2mm]
{\bf Keywords:} Boolean Gallai-Ramsey number; Boolean rainbow Ramsey number; Poset; Boolean lattice \\[2mm]
{\bf AMS subject classification 2020:} 05C55; 05D05; 05D10; 06A07.
\end{abstract}

\section{Research Background}
A \emph{partially ordered set}, or \emph{poset}, is a pair $(\mathcal{P}, \leq)$, where $\mathcal{P}$ is a family of sets and $\leq$ is a relation that is reflexive, antisymmetric, and transitive. The size of a poset $\mathcal{P}$ is denoted by $|\mathcal{P}|$.
For integers $m \le n$, we write $[m,n] = \{m, m+1, \dots, n\}$; in particular, we abbreviate $[1,n]$ to $[n]$.
The \emph{Boolean lattice of dimension $n$}, denoted $\mathcal{B}_n = (2^{[n]}, \leq)$, is the power set of an $n$-element ground set $[n]$ ordered by inclusion; note that $\mathcal{B}_0 = \{\emptyset\}$. A pair $X, Y \in \mathcal{P}$ is said to be \emph{comparable}, denoted $X \sim Y$, if $X \leq Y$ or $Y \leq X$. We write $X < Y$ if $X \leq Y$ and $X \neq Y$. If $X \not< Y$ and $Y \not< X$, then $X$ and $Y$ are \emph{incomparable}, denoted $X \nsim Y$.
A poset $\mathcal{C}_t$, written as $(X_1,X_2,\ldots,X_t)$ when its sets are specified, is called a \emph{$t$-chain} if any two sets in $\mathcal{C}_t$ are comparable.
Similarly, a poset $\mathcal{A}_t$ with $|\mathcal{A}_t| = t$ is called a
\emph{$t$-antichain} if any two distinct sets of $\mathcal{A}_t$ are incomparable.

For two posets $\mathcal{P}$ and $\mathcal{Q}$, a subset $\mathcal{X} \subseteq \mathcal{Q}$ is called a \emph{weak copy} of $\mathcal{P}$ if there exists an injection $f : \mathcal{P} \to \mathcal{X}$ such that $f(X) \le f(Y)$ in $\mathcal{Q}$ whenever $X \le Y$ in $\mathcal{P}$. If, in addition, we have $f(X) \le f(Y)$ in $\mathcal{Q}$ if and only if $X \le Y$ in $\mathcal{P}$, then $\mathcal{X}$ is called an \emph{induced} copy of $\mathcal{P}$. We say that $\mathcal{Q}$ \emph{contains} a weak copy (respectively, an induced copy) of $\mathcal{P}$ if it has a subset which is a weak copy (respectively, an induced copy) of $\mathcal{P}$. A colored poset $\mathcal{P}$ is \textit{monochromatic} if all of its sets
share the same color. A colored poset $\mathcal{P}$ is \textit{rainbow} if all of its sets have distinct colors. A $k$-coloring of $\mathcal{P}$ is called \emph{exact} if every color in $[k]$ is used at least once on $\mathcal{P}$.

Ramsey theory on posets was initiated by Ne\v{s}et\v{r}il and R\"{o}dl~\cite{NesetrilRodl}, focusing on induced copies of posets. Kierstead and Trotter~\cite{KiersteadTrotter} extended this line of research by considering arbitrary host posets beyond the Boolean lattice, and investigated Ramsey-type problems in terms of cardinality, height, and width. For further details and developments, we refer the reader to~\cite{AW23, BP2021, DKT91, GMT23, KiersteadTrotter, LT22, McColm, TrotterRamsey, Walzer15}.

In the setting of Boolean lattices, Axenovich and Walzer~\cite{AW17} initiated the study of 
Ramsey-type problems on $\mathcal{B}_n$ and introduced the corresponding Ramsey parameters.
Given a poset $\mathcal{P}$ and an integer $k \ge 2$, the \emph{Boolean Ramsey number}
$\operatorname{R}_k(\mathcal{P})$ is defined as the smallest integer $n$
such that any $k$-coloring of the sets of $\mathcal{B}_n$ contains a monochromatic induced copy of $\mathcal{P}$.
There is also a natural weak version. Following Cox and Stolee~\cite{CS18}, the
\emph{weak Boolean Ramsey number} $\operatorname{R}_k^{\text{w}}(\mathcal{P})$
is defined as the smallest integer $n$ such that any $k$-coloring of $\mathcal{B}_n$ contains a
monochromatic weak copy of $\mathcal{P}$.

Chen, Cheng, Li, and Liu~\cite{CCLL20} introduced the Boolean rainbow Ramsey numbers 
for posets. For posets $\mathcal{Q}$ and $\mathcal{P}$, the \emph{Boolean rainbow Ramsey 
number} $\operatorname{RR}(\mathcal{Q}:\mathcal{P})$ is defined as the smallest 
integer $n$ such that any coloring of the sets of $\mathcal{B}_n$ contains either a rainbow induced copy of $\mathcal{Q}$ or a
monochromatic induced copy of $\mathcal{P}$.
They determined the exact values of Boolean rainbow 
Ramsey numbers when both $\mathcal{P}$ and $\mathcal{Q}$ are antichains, Boolean lattices, 
or chains.
In analogy with the weak Ramsey numbers above, one may also consider a weak rainbow
version: for posets $\mathcal{Q}$ and $\mathcal{P}$, the \emph{weak Boolean rainbow
Ramsey number} $\operatorname{RR}^{\text{w}}(\mathcal{Q}:\mathcal{P})$ is the smallest integer
$n$ such that any coloring of $\mathcal{B}_n$ contains either a rainbow copy of $\mathcal{Q}$ or a monochromatic copy of
$\mathcal{P}$.
Weak rainbow Ramsey numbers and rainbow Ramsey numbers for Boolean lattices were systematically studied by
Chang, Gerbner, Li, Methuku, Nagy, Patk\'os, and Vizer~\cite{CGLMNPV22}.

The parameters $\operatorname{RR}(\mathcal{Q}:\mathcal{P})$ and 
$\operatorname{RR}^{\text{w}}(\mathcal{Q}:\mathcal{P})$ measure which monochromatic or rainbow 
configurations must appear in arbitrary colorings of $\mathcal{B}_n$. 
However, they do not explicitly capture the additional structural restrictions on colorings 
that arise when the number of colors is fixed and rainbow copies of a given pattern are forbidden. 
This motivates the introduction of a parameter that is sensitive to these restrictions on colorings.

Precisely this phenomenon has been systematically investigated in the graph setting under the 
name of Gallai-Ramsey theory. 
The Gallai-Ramsey theory for graphs goes back to Gallai's work ~\cite{Gallai} in 1967, 
where he observed that colorings of complete graphs avoiding a rainbow triangle 
exhibit a remarkably structured and somewhat surprising behavior. 
Given two graphs $G$ and $H$, the \emph{Gallai-Ramsey number} 
$\operatorname{gr}_k(G:H)$ is defined as the smallest integer $n$ such that for every 
$N \ge n$, any exact $k$-coloring of the edges of the complete graph on $N$ vertices 
contains either a rainbow copy of $G$ or a monochromatic copy of $H$. 
For a comprehensive and up-to-date survey of Gallai-Ramsey numbers, we refer the reader to~\cite{FMO14}. More recently, 
variants such as the integer Gallai-Ramsey theory and the Euclidean Gallai-Ramsey theory 
have been introduced and studied in~\cite{CX24, MORW23, MOW22, MP24}.

Guided by the Boolean rainbow Ramsey theory and the graph Gallai-Ramsey theory, we introduce the notion of Boolean Gallai-Ramsey numbers for posets.
Before giving the formal definition of the Boolean Gallai-Ramsey number, we introduce a term that streamlines later exposition.
Accordingly, we say that the pair of posets \((\mathcal{Q},\mathcal{P})\) is \emph{\((\mathcal{Q}:\mathcal{P})_{k}\)-good} if for every positive integer \(N\), every exact \(k\)-coloring of \(\mathcal{B}_N\) either contains a rainbow induced copy of \(\mathcal{Q}\) or contains a monochromatic induced copy of \(\mathcal{P}\).
In general, this property need not hold for all \(N\). We therefore introduce the Boolean Gallai-Ramsey number.

\begin{definition}
For posets \(\mathcal{Q},\mathcal{P}\), which are not $(\mathcal{Q},\mathcal{P})_k$-good, the \emph{Boolean Gallai-Ramsey number}
\(\operatorname{GR}_{k}(\mathcal{Q}:\mathcal{P})\)
is defined to be the smallest integer \(n\) such that for every \(N \ge n\), every exact \(k\)-coloring of \(\mathcal{B}_N\) contains either a rainbow induced copy of \(\mathcal{Q}\) or a monochromatic induced copy of \(\mathcal{P}\).
\end{definition}

One can similarly define weak Boolean Gallai-Ramsey numbers by replacing induced copies with
weak copies; in this paper we restrict attention to the induced version.

\subsection{Our results}
In this paper we obtain structural characterizations of exact colorings of Boolean
lattices that avoid rainbow induced copies of small posets, and we apply these
characterizations to both Gallai-Ramsey type problems and rainbow Ramsey
type problems for posets.

\medskip\noindent\textbf{(1) Structural characterizations.}
In Section~\ref{sec_2} we prove three structural results on exact $k$-colorings
of Boolean lattices that avoid rainbow induced copies of small posets.
More precisely, we obtain characterizations for colorings without rainbow
induced copies of the chain $\mathcal{C}_3$, the fork
$\vee_2 = (\{X_0,X_1,X_2\},\le)$ with $X_0 < X_1$, $X_0 < X_2$, and
$X_1 \nsim X_2$, and the $2$-dimensional Boolean lattice $\mathcal{B}_2$;
see Theorems~\ref{Structural-C3-1}, \ref{Structural-V2-1}, and
\ref{Structural-B2-1}. These structural descriptions are the main tools
for all of our applications.

\medskip\noindent\textbf{(2) Applications to Boolean Gallai-Ramsey numbers.}
In Section~\ref{sec_3} we apply the above structural theorems to determine or
bound Boolean Gallai-Ramsey numbers for several pairs of posets.

Our first Gallai-Ramsey result concerns chains of length $3$ and $s$,
based on the structural characterization given in Theorem~\ref{Structural-C3-1}.

For chains of lengths $3$ and $s$, we use the structural characterization given in
Theorem~\ref{Structural-C3-1} to obtain the following theorem.

\begin{theorem}\label{cBGR-c3}
Let $k, s$ be integers with $k\geq 3$ and $s\geq 3$. Then 
\begin{itemize}
    \item[] $(1)$ $\mathrm{GR}_{k}(\mathcal{C}_3:\mathcal{C}_s)=s$ if $3\leq k\leq{s-1\choose \lceil (s-1)/2\rceil}+1$.
    
    \item[] $(2)$ The pair of posets $\mathcal{C}_3, \mathcal{C}_s$ is $(\mathcal{C}_3:\mathcal{C}_s)_k$-good if $k>{s-1\choose \lceil (s-1)/2\rceil}+1$.
\end{itemize}
\end{theorem}

Using Theorem~\ref{Structural-V2-1}, we obtain analogous results for the fork
$\vee_2$ versus chains.

\begin{theorem}\label{cBGR-v2}
Let $k, s$ be two integers with $k\geq 3$ and $s\geq 2$. Then 
\begin{itemize}
    \item[] $(1)$ $\mathrm{GR_{3}}(\vee_2:\mathcal{C}_s)=2s-1$.
    \item[] $(2)$ The pair of \(\vee_2\) and \(\mathcal{C}_s\) is \((\vee_2:\mathcal{C}_s)_4\)-good for \(s\in\{2,3\}\), and \(\operatorname{GR}_{4}(\vee_2:\mathcal{C}_s)=s\) for all \(s\ge 4\).
    \item[] $(3)$ The pair of $\vee_2, \mathcal{C}_s$ is $(\vee_2:\mathcal{C}_s)_k$-good if $k\geq 5$.
\end{itemize}
\end{theorem}

Our third Gallai-Ramsey result concerns Boolean lattices $\mathcal{B}_2$ and
$\mathcal{B}_n$ and relies on the structural description in Theorem~\ref{Structural-B2-1}.

\begin{theorem}\label{gr2m}
Let $k, n$ be two integers with $k\geq 4$ and $n\geq 1$. 
\begin{itemize}
    \item[] $(1)$ If $4\leq k\leq 2^{\mathrm{R}_3\left(\mathcal{B}_n\right)+n}$, then $\mathrm{GR}_k(\mathcal{B}_2:\mathcal{B}_n) \leq \mathrm{R}_3\left(\mathcal{B}_n\right)+n.$
    \item[] $(2)$ If $k>2^{\mathrm{R}_3\left(\mathcal{B}_n\right)+n}$, then the pair of posets $\mathcal{B}_2, \mathcal{B}_n$ is $(\mathcal{B}_2:\mathcal{B}_n)_k$-good.
\end{itemize} 
\end{theorem}

\medskip\noindent\textbf{(3) Applications to Boolean rainbow Ramsey numbers.}
In Section~\ref{sec_4} we study rainbow Ramsey numbers for Boolean lattices.

Our first result considers rainbow copies of Boolean lattices.
Motivated by the Blob Lemma of Axenovich and Walzer (\cite{AW17}, Lemma~3), we establish a sufficient condition under which an exact \(k\)-coloring of \(\mathcal{B}_N\) contains no rainbow induced copy of \(\mathcal{B}_m\), and we use it to obtain the following upper bound.

\begin{theorem}\label{rainbow-Boolean-m-n}
Let $m, n$ be positive integers. Then 
$\operatorname{RR}(\mathcal{B}_m:\mathcal{B}_{n})\leq m\,\operatorname{R}_{2^m-1}(\mathcal{B}_{n})+m$.
\end{theorem}

{For comparison, we recall upper bounds due to Chen, Cheng, Li, and Liu~\cite{CCLL20}. They proved in~\cite[Theorem~14]{CCLL20} that \(\mathrm{RR}(\mathcal{B}_m:\mathcal{B}_n)\le \sum_{i=1}^{2^m-1}\mathrm{R}_i(\mathcal{B}_n)\). Moreover, by combining this with the bound \(\mathrm{R}_k(\mathcal{B}_m)\le 1000\,m^{7}16^{m}k\), they obtained \(\mathrm{RR}(\mathcal{B}_m:\mathcal{B}_n)<2^{2m+4n+9}n^7\); see~\cite[p.~11]{CCLL20}.

A sharper upper bound of $\mathrm{R}_k(\mathcal{B}_n)$ is available due to Cox and Stolee~\cite[Corollary~2.14]{CS18}, namely \(\mathrm{R}_k(\mathcal{B}_n)\le 2kn^2\). Combining this with \(\mathrm{RR}(\mathcal{B}_m:\mathcal{B}_n)\le \sum_{i=1}^{2^m-1}\mathrm{R}_i(\mathcal{B}_n)\) yields \(\mathrm{RR}(\mathcal{B}_m:\mathcal{B}_n)\le 2^m(2^m-1)n^2\).}
In contrast, combining Theorem~\ref{rainbow-Boolean-m-n} with $\mathrm{R}_k(\mathcal{B}_n)\le 2kn^2$ gives
$\mathrm{RR}(\mathcal{B}_m:\mathcal{B}_n)\le 2m(2^m-1)n^2 + m.$
Thus, while both bounds have the same leading term in $n$, our estimate improves
the dependence on $m$ from $2^m$ down to $2m$ for $m \ge 3$, and it is also
substantially smaller than the explicit bound
$2^{2m+4n+9}n^7$ from~\cite{CCLL20}.

For $m=2$, by combining Theorem~\ref{gr2m} with a trivial lower bound we obtain the second result for Boolean rainbow Ramsey numbers, which further sharpen the upper bound in Theorem~\ref{rainbow-Boolean-m-n}.
\begin{theorem}\label{th-Bm2}
For the Boolean lattice $\mathcal{B}_n$ with $n\geq 1$, we have
$\mathrm{R}_3\left(\mathcal{B}_n\right)\leq \mathrm{RR}(\mathcal{B}_2:\mathcal{B}_n) \leq \mathrm{R}_3\left(\mathcal{B}_n\right)+n.$
\end{theorem}

{When \(m=2\), the bound of Chen, Cheng, Li, and Liu~\cite{CCLL20} specializes to
\(\mathrm{RR}(\mathcal{B}_2:\mathcal{B}_n)\le \sum_{i=1}^{3}\mathrm{R}_i(\mathcal{B}_n)
= n+\mathrm{R}_2(\mathcal{B}_n)+\mathrm{R}_3(\mathcal{B}_n)\).
In contrast, our upper bound eliminates the \(\mathrm{R}_2(\mathcal{B}_n)\) term.}

For a poset $\mathcal{P}$, a family $\mathcal{F} \subseteq \mathcal{B}_n$ is said to be induced $\mathcal{P}$-free if it does not contain an induced copy of $\mathcal{P}$.
The \emph{$k$-th level} of $\mathcal{B}_n$ is the collection of all $k$-element subsets of $[n]$, denoted by ${[n] \choose k}$.
Let $e(\mathcal{P})$ denote the maximum integer $m$ such that, for any $n$, the union of any $m$ consecutive levels of $\mathcal{B}_n$ is induced $\mathcal{P}$-free.
A common tool in studying Tur\'{a}n-type questions in posets is the Lubell function. For a family $\mathcal{F}\subseteq \mathcal{B}_n$, the \emph{Lubell function} of $\mathcal{F}$ is defined as
$\operatorname{lu}_n(\mathcal{F}) = \sum_{F\in \mathcal{F}} {n\choose |F|}^{-1}$.
For a poset $\mathcal{P}$, let $\operatorname{Lu}_n(\mathcal{P})$ denote the maximum value of $\operatorname{lu}_n(\mathcal{F})$ over all induced $\mathcal{P}$-free families $\mathcal{F} \subseteq \mathcal{B}_n$.
By the definitions of $e(\mathcal{P})$, we have
$e(\mathcal{P}) \leq \operatorname{Lu}_n(\mathcal{P})$ for every poset $\mathcal{P}$ and every integer $n \geq e(\mathcal{P})$.
We say that a poset $\mathcal{P}$ is  \emph{uniformly induced Lubell-bounded} if $e(\mathcal{P}) \geq \operatorname{Lu}_n(\mathcal{P})$ for all positive integers $n$. An example of posets satisfying the latter property is the chain $\mathcal{C}_\ell$.

Our third rainbow result concerns the poset $\vee_2$ and uses the
structure from Theorem~\ref{Structural-V2-1}.
\begin{theorem}\label{v2-p}
Let $\mathcal{P}$ be a uniformly induced Lubell-bounded poset, other than $\mathcal{C}_1$. Then
\[
\mathrm{RR}(\vee_2:\mathcal{P})=2e(\mathcal{P})+1.
\]
\end{theorem}

Chang, Gerbner, Li, Methuku, Nagy, Patk\'os, and Vizer~\cite{CGLMNPV22} showed that for all posets \(\mathcal{P}\) and \(\mathcal{Q}\) one has the general lower bound
\(\operatorname{RR}(\mathcal{Q}:\mathcal{P}) \ge e(\mathcal{P})(|\mathcal{Q}| - 1) + g(\mathcal{Q})\),
where \(g(\mathcal{Q})\) is defined by \(g(\mathcal{Q}) = 0\) if \(\mathcal{Q}\) has both a maximum and a minimum element, \(g(\mathcal{Q}) = 2\) if \(\mathcal{Q}\) has neither a maximum nor a minimum element, and \(g(\mathcal{Q}) = 1\) otherwise. They proved that this bound is tight in the case \(\mathcal{Q}=\mathcal{A}_3\), provided that \(\mathcal{P}\) is uniformly induced Lubell-bounded and \(\mathcal{P}\notin\{\mathcal{C}_1,\mathcal{C}_2\}\). They also constructed pairs \((\mathcal{Q},\mathcal{P})\) for which the inequality is strict, for instance with \(\mathcal{P}=\mathcal{C}_2\) and \(\mathcal{Q}=\mathcal{A}_k\) for \(k\ge 2\). These results naturally led them to ask the following question.

\begin{question}[\cite{CGLMNPV22}]\label{q1}
For which uniformly induced Lubell-bounded posets $\mathcal{P}$ does the equality
\[
\operatorname{RR}(\mathcal{Q}:\mathcal{P})
= e(\mathcal{P})(|\mathcal{Q}| - 1) + g(\mathcal{Q})
\]
hold for every poset $\mathcal{Q}$?
\end{question}

Theorem \ref{v2-p} confirms the formula proposed in Question~\ref{q1} for the case
$\mathcal{Q}=\vee_2$ and all uniformly induced Lubell-bounded posets
$\mathcal{P}$, other than $\mathcal{C}_1$.

\section{Structure of colorings without small rainbow subposets}\label{sec_2}

In this section, we present several structural theorems for Boolean lattices that do not contain a rainbow induced copy of the subposets $\mathcal{C}_3$, $\vee_2$, and $\mathcal{B}_2$. 
We first introduce some notation. 
For sets $X,Y \in \mathcal{B}_n$ with $X \le Y$, we use the interval notation
$\mathcal{B}_{[X, Y]} = \{Z : X \le Z \le Y\}$, $\mathcal{B}_{(X, Y]} = \{Z : X < Z \le Y\}$, $\mathcal{B}_{[X, Y)} = \{Z : X \le Z < Y\}$, and $\mathcal{B}_{(X, Y)} = \{Z : X < Z < Y\}$.
Moreover, if $\mathcal{P}$ and $\mathcal{Q}$ are posets with $\mathcal{P}\subseteq \mathcal{Q}$, we write
$\mathcal{Q} \setminus \mathcal{P}$ for the poset obtained from $\mathcal{Q}$ by deleting all sets of $\mathcal{P}$.

\subsection{Characterization of exact $k$-colorings without a rainbow induced copy of $\mathcal{C}_3$}

\begin{theorem}\label{Structural-C3-1}
Let $k, n$ be two integers with $n\geq 2$ and $k\geq 3$. Consider an exact $k$-coloring $c$ of the sets in $\mathcal{B}_n$. Then there is no rainbow induced copy of $\mathcal{C}_3$ if and only if $c(\emptyset) = c([n])$ and for any two sets $X, Y\in \mathcal{B}_{(\emptyset, [n])}$, we have $X\nsim Y$ whenever $c(X), c(Y), c(\emptyset)$ are pairwise distinct.
\end{theorem}

\begin{proof}
We first show the ``if'' part. Suppose that $c(\emptyset)=c([n])$, and for any $X, Y\in \mathcal{B}_{(\emptyset, [n])}$ we have $X\nsim Y$ whenever $c(X), c(Y), c(\emptyset)$ are pairwise
distinct.  
Assume for contradiction that there exists a rainbow induced copy of $\mathcal{C}_3$, say $(X, Y, Z)$.  
If $\emptyset\notin \{X, Y, Z\}$, then there exist two sets in $\{X, Y, Z\}$, say $X$ and $Y$, such that $c(X)\neq c(\emptyset)$ and $c(Y)\neq c(\emptyset)$. 
Since $c(X)\neq c(Y)$, it follows that $X\nsim Y$, which contradicts the fact that $X\sim Y$ in the rainbow induced copy of $\mathcal{C}_3$.  
If $\emptyset\in \{X, Y, Z\}$, then assume without loss of generality that $Z=\emptyset$. Since $(\emptyset, X, Y)$ is a rainbow induced copy of $\mathcal{C}_3$, it follows that $c(X), c(Y)$, and $c(\emptyset)$ are pairwise distinct, and hence $X\nsim Y$, a contradiction.  
Thus, no rainbow induced copy of $\mathcal{C}_3$ can exist.

We next prove the ``only if" part. Suppose that there is no rainbow induced copy of $\mathcal{C}_3$. Without loss of generality, assume $c(\emptyset)=1$. 
If $c([n])\neq c(\emptyset)$, then since $k\geq 3$, there exists a set $X$ with $\emptyset<X<[n]$ such that $c(X)\neq c(\emptyset)$ and $c(X)\neq c([n])$. This means that $(\emptyset, X, [n])$ is a rainbow induced copy of $\mathcal{C}_3$, which is a contradiction. Therefore, it must hold that $c(\emptyset) = c([n])$. 
Now, consider any two sets $X, Y\in \mathcal{B}_{(\emptyset, [n])}$ such that $c(X), c(Y)\in [2, k]$ and $c(X)\neq c(Y)$. Since $\{\emptyset, X, Y\}$ cannot form a rainbow induced copy of $\mathcal{C}_3$, it follows that $X\nsim Y$, completing the proof.
\end{proof}

Combining Theorem~\ref{Structural-C3-1} with the fact that the largest antichain in $\mathcal{B}_n$ has size $\binom{n}{\lceil n/2\rceil}$, we obtain the following corollary.

\begin{corollary}\label{cor:upper-bound-k}
Let $n\geq 2$ and $k\geq 3$. Consider an exact $k$-coloring of $\mathcal{B}_n$ with no rainbow induced copy of $\mathcal{C}_3$. Then $k\leq \binom{n}{\lceil n/2\rceil}+1$.
\end{corollary}

\subsection{Characterization of exact $k$-colorings without a rainbow induced copy of $\vee_2$}
\begin{theorem}\label{Structural-V2-1}
Let $k, n$ be two integers with $k\geq 3$ and $n\geq 2$. Considering an exact $k$-coloring of $\mathcal{B}_n$, there is no rainbow induced copy of $\vee_2$ if and only if one of the following holds:
\begin{itemize}
    \item[$(1)$] There exists a set $A$ with $|A|\leq n-2$ such that
    all sets in the families $\mathcal{B}_{(A, [n])}, \mathcal{B}_n\setminus \mathcal{B}_{[A, [n]]}$, and $\{A\}$ are monochromatically colored so that no two sets from different families share the same color.
If $k=3$, then the color of $[n]$ is unrestricted; if $k=4$, then $[n]$ receives a color different from those in $\mathcal{B}_{[\emptyset, [n])}$.
    \item[$(2)$] Exactly two colors appear in $\mathcal{B}_{[\emptyset, [n])}$. The set $[n]$ receives a color, distinct from the two colors used in $\mathcal{B}_{[\emptyset, [n])}$.
\end{itemize}
\end{theorem}

Note that at most $3$ colors can appear in $\mathcal{B}_{[\emptyset, [n])}$. Before proving Theorem \ref{Structural-V2-1}, we show the following lemma, which is useful for our proofs.

\begin{lemma}\label{lem:TL1}
Let $n$ be an integer with $n \geq 2$. Let $\emptyset\leq X< Z\leq [n]$ with $|Z|\geq |X|+2$. If $W\in \mathcal{B}_n$ satisfies $W\notin \mathcal{B}_{[\emptyset, X]}\cup \mathcal{B}_{[Z, [n]]}$, 
then there exists $W_0\in \mathcal{B}_{(X, Z)}$ with $|W_0|=|X|+1$ such that $W_0\nsim W$.
\end{lemma}
\begin{proof}
Since $W\notin \mathcal{B}_{[\emptyset, X]}\cup \mathcal{B}_{[Z, [n]]}$, at least one of $X<W<Z$, $W\nsim X$, or $W\nsim Z$ holds.
Suppose $X<W<Z$, then choose $a\in W\setminus X$ and $t\in Z\setminus W$.
As $X\subset W$, we have $Z\setminus W\subset Z\setminus X$, and hence $t\in Z\setminus X$.
Set $W_0=X\cup\{t\}\in \mathcal{B}_{(X, Z)}$.
Moreover, $t\in W_0\setminus W$, and thus $W_0\not\leq W$.
Also $a\notin W_0$ (since $a\notin X$ and $a\neq t$ because $t\notin W$), and hence $W\not\leq W_0$.
Therefore, $W_0\nsim W$, as required.

Suppose $W\nsim X$.    
Then there exist elements
$a\in W\setminus X$ and $b\in X\setminus W$.
If $(Z\setminus X)\setminus W\neq\emptyset$, then
pick $t\in (Z\setminus X)\setminus W$ and set
$W_0=X\cup\{t\}$.
Thus, $W_0\in \mathcal{B}_{(X, Z)}$.
Since $a\in W\setminus W_0$ and $t\in W_0\setminus W$, we have $W_0\nsim W$. 
If $(Z\setminus X)\setminus W=\emptyset$, then $Z\setminus X\subseteq W$.
Since $|Z\setminus X|\ge 2$, choose two distinct elements $t, s\in Z\setminus X$ and let $W_0=X\cup\{t\}$, and thus $W_0\in \mathcal{B}_{(X, Z)}$.
Now $s\in W\setminus W_0$ (since $s\neq t$ and $Z\setminus X\subseteq W$) and $b\in W_0\setminus W$ (because $b\in X\setminus W\subseteq W_0$), and hence $W_0\nsim W$.

Suppose $W\nsim Z$.  
Choose $a\in W\setminus Z$ and $b\in Z\setminus W$. By the same argument in the previous paragraph, we have $W_0\nsim W$.
\end{proof}

\begin{proof}[Proof of Theorem~\ref{Structural-V2-1}]
We first show the ``if'' part.
Suppose that an exact $k$-coloring $c$ of the sets in $\mathcal{B}_n$ satisfies $(1)$. Without loss of generality, let $c(\emptyset)=1$, $c(A)=2$,  and $c(X)=3$ for each $X\in \mathcal{B}_{(A, [n])}$, 
where $A$ is the set described in (1).
Since $[n]$ is comparable with every set in $\mathcal{B}_n$, the set $[n]$ cannot be contained in any induced copy of $\vee_2$. 
Suppose that there exists a rainbow induced copy of $\vee_2$ with colors $1, 2, 3$.
Since $A$ is the only set of color $2$, $A$ has to be contained in the rainbow induced copy of $\vee_2$.
Let $\{A, W, Y\}$ form a rainbow induced copy of $\vee_2$, where $c(W)=1$ and $c(Y)=3$.
Since the sets with color $3$ only appear in $\mathcal{B}_{(A, [n])}$, it follows that $A < Y$.
Thus, $A < W$.
However, this is impossible since the sets with color $1$ only appear in $\mathcal{B}_n\setminus \mathcal{B}_{[A, [n]]}$.
Therefore, there exists no rainbow induced copy of $\vee_2$.

In the case when an exact $k$-coloring $c$ of the sets in $\mathcal{B}_n$ satisfies $(2)$, exactly two colors appear in $\mathcal{B}_{[\emptyset, [n])}$. Recall that the set $[n]$ cannot be contained in any induced copy of $\vee_2$. Therefore, every induced copy of $\vee_2$ must be contained in $\mathcal{B}_{[\emptyset, [n])}$ and hence there exists no rainbow induced copy of $\vee_2$.

We next prove the ``only if'' part. Suppose that there is no rainbow induced copy of $\vee_2$. Let $k'$ be the number of colors used in $\mathcal{B}_{[\emptyset, [n])}$. If $k'=2$, then $(2)$ holds and we are done with this part. Therefore, we may assume that $k'\ge 3$.

{If \(n=2\), then since \(k'\ge 3\), there must exist a rainbow \(\vee_2\), a contradiction. Thus \(n\ge 3\).}
Without loss of generality, let $c(\emptyset)=1$. Moreover, every set in $\mathcal{B}_{[\emptyset,[n])}$ is assigned a color from $[k']$. Let $X, Y\in \mathcal{B}_{(\emptyset, [n])}$ be sets with $c(X), c(Y)\in [2, k']$ and $c(X)\neq c(Y)$. If $X\nsim Y$, then the triple $\{\emptyset, X, Y\}$ forms a rainbow induced copy of $\vee_2$ whose three colors are $1$, $c(X)$, and $c(Y)$, a contradiction. Therefore for any such $X$ and $Y$, we must have $X\sim Y$.

Let $A\in \mathcal{B}_{[\emptyset, [n])}$ be a set of minimum size such that $c(A)\neq 1$. By relabeling colors if necessary, we may and do assume that $c(A)=2$. By the minimality of $A$, for every set $X$ with $\emptyset\le X<A$ we have $c(X)=1$. Since $k'\ge 3$, there exist sets $X_3, \dots, X_{k'}\in \mathcal{B}_{[\emptyset, [n])}$ such that $c(X_i)=i$ for every $i$ with $3\le i\le k'$. From the conclusion of the previous paragraph, for all distinct indices $i$ and $j$ with $3\le i\neq j\le k'$ we have $X_i\sim X_j$, and for every $i$ with $3\le i\le k'$ we have $A< X_i$. By renaming these sets if necessary, we may assume that $A<X_3<\cdots<X_{k'}<[n]$.

We now show that $k'=3$. Suppose $k'\geq 4$.  Since $A<X_3<X_4<[n]$, it follows from Lemma \ref{lem:TL1} with $X=A$, $Z=[n]$, and $W\in\{X_3, X_4\}$ that there exist $X'_3$, $X'_4\in \mathcal{B}_{(A, [n])}$ with $|X'_3|=|X'_4|=|A|+1$ such that $X'_3\nsim X_3$ and $X'_4\nsim X_4$.
Since $c(\emptyset)=1$, $c(A)=2$, and neither $\{\emptyset, X_i, X'_i\}$ nor $\{A, X_i, X'_i\}$ forms a rainbow induced copy of $\vee_2$, 
we have $c(X'_i)=i$ for $i=3, 4$, and hence $X'_3\neq X'_4$. Note that $X'_3, X'_4\in \mathcal{B}_{(A, [n])}$ and $|X'_3|=|X'_4|$. Therefore, $\{A, X'_3, X'_4\}$ forms a rainbow induced copy of $\vee_2$, a contradiction. Thus, $k'\leq 3$, and hence $k'=3$.

Consequently, there exists a set $X_3$ such that $A<X_3<[n]$ and $c(X_3)=3$. By the fact obtained in the previous paragraph, there is a set $X'_3\in \mathcal{B}_{(A, [n])}$ with $|X'_3|=|A|+1$ such that $c(X'_3)=3$. 
For each $W\in \mathcal{B}_{(A, [n])}$ with $|W|=|A|+1$ and $W\neq X'_3$, since neither $\{\emptyset, W, X'_3\}$ nor $\{A, W, X'_3\}$ forms a rainbow induced copy of $\vee_2$, we have $c(W)=3$. 
We then claim that $c(Y)=3$ for each $Y \in \mathcal{B}_{(A, [n])}$. Let $Y\in \mathcal{B}_{(A, [n])}$. By Lemma \ref{lem:TL1} with
$X=A$, $Z=[n]$, and $W=Y$, there exists $W_0\in \mathcal{B}_{(A, [n])}$ with $|W_0|=|A|+1$ such that $Y\nsim W_0$.
Since neither $\{\emptyset, Y, W_0\}$ nor $\{A, Y, W_0\}$ forms a rainbow induced copy of $\vee_2$, 
we have $c(Y)=c(W_0)=3$, as claimed.

In the end, we prove that $c(Y)=1$ for each $Y\in \mathcal{B}_n\setminus \mathcal{B}_{[A, [n]]}$.
Let $Y\in \mathcal{B}_n\setminus \mathcal{B}_{[A, [n]]}$.
If $Y \in \mathcal{B}_{[\emptyset, A)}$, 
then it follows from the choice of $A$ that $c(Y) = 1$.
Thus, we may assume $Y \notin \mathcal{B}_{[\emptyset, A)}$, 
and hence $Y \nsim A$.
{By Lemma \ref{lem:TL1} with
$X=A$, $Z=[n]$, and $W=Y$, there exists $W_0\in \mathcal{B}_{(A, [n])}$ with $|W_0|=|A|+1$ such that $Y\nsim W_0$.}
If $c(Y) \neq 1, 2$, then $\{\emptyset, A, Y\}$ forms a rainbow induced copy of $\vee_2$, a contradiction.
{If $c(Y)=2$, then then $\{\emptyset, W_0, Y\}$ forms a rainbow induced copy of $\vee_2$, a contradiction.}

Therefore, $c(X)=1$ for each $X\in \mathcal{B}_n\setminus \mathcal{B}_{[A, [n]]}$.
Then $(1)$ holds. 
This completes the proof of the only if part.
\end{proof}

%%%%%%%%%%%%%%%%%%%%%%%%%%%%%%%%%%%%%%%%%%%%%%%%%%%%%%%%%%%%%%%%%%%%%%%%%%%
\subsection{An exact $5$-coloring $c$ without a rainbow induced copy of $\mathcal{B}_2$ with $c(\emptyset)\neq c([n])$}
%%%%%%%%%%%%%%%%%%%%%%%%%%%%%%%%%%%%%%%%%%%%%%%%%%%%%%%%%%%%%%%%%%%%%%%%%%%

We now turn to a characterization of  coloring of $\mathcal{B}_n$ without a rainbow induced copy of $\mathcal{B}_2$.
{We first consider the case \(c(\emptyset)\neq c([n])\). In this section, we show that when the number of colors $k$ satisfies \(k\ge 5\), the only possible coloring is of Type~1, and in fact one must have \(k=5\). In Section \ref{subsection24}, we then treat the case \(k=4\).
Next, in Section \ref{subsection25} we consider the case \(c(\emptyset)=c([n])\).
Finally, we  state the complete characterization as Theorem \ref{Structural-B2-1} in Section \ref{subsection26}.}

{We first introduce a basic exact \(5\)-coloring of \(\mathcal{B}_n\) of \emph{Type~1} as follows.}

\begin{itemize}
\item
%[$(1)$] 
\emph{Type $1$}: There exist two sets $X_0, Y_0 \in \Bn$ with $\emptyset < X_0 < Y_0 < [n]$ and $|Y_0| \geq |X_0| + 2$
such that all sets in the following families are monochromatically colored so that no two sets from different families share the same color:
$$
\mathcal{B}_{[\emptyset, Y_0]} \setminus \mathcal{B}_{[X_0, Y_0]}, \quad \{X_0\}, \quad \mathcal{B}_{(X_0, Y_0)}, \quad \{Y_0\}, \quad \mathcal{B}_{[X_0, [n]]} \setminus \mathcal{B}_{[X_0, Y_0]}.$$
Moreover, each set in $\mathcal{B}_{n} \setminus \left(\mathcal{B}_{[\emptyset, Y_0]}\cup \mathcal{B}_{[X_0, [n]]}\right)$ has the same color as the sets in the first or the last
family.
\end{itemize}

As we will explain below, 
it is actually the only exact $k$-coloring of $\Bn$ for $k \geq 5$
such that $\emptyset$ and $[n]$ are assigned different colors
and there is no rainbow induced copy of $\BB$.
Note that an exact $5$-coloring of Type 1
is symmetric by exchanging the roles of $X_0$ and $Y_0$, 
those of $\emptyset$ and $[n]$, and so on.
Both $\mathcal{B}_{[\emptyset, Y_0]}$ and $\mathcal{B}_{[X_0, [n]]}$ are colored by $4$ colors.
Assuming $|Y_0|\ge |X_0|+2$ ensures $\mathcal{B}_{(X_0, Y_0)}\neq\emptyset$, which will be used repeatedly below.

\begin{lemma}\label{k5}
Let $k, n$ be integers with $k\geq 5$ and $n\geq 4$. 
Consider an exact $k$-coloring $c$ of $\mathcal{B}_n$
with $c(\emptyset)\neq c([n])$.
%in which $\emptyset$ and $[n]$ are assigned different colors.
%where $n \geq 3$. 
Then there is no rainbow induced copy of $\mathcal{B}_2$ if and only if $k=5$ and the coloring $c$ is of Type 1.
\end{lemma}

Before proving Lemma \ref{k5}, 
we show the following lemma, which is useful for our proofs.

\begin{lemma}\label{lem-x-y}
Let $k, n$ be two integers with $k\geq 4$ and $n\geq 2$. 
Consider an exact $k$-coloring $c$ of $\mathcal{B}_n$ with $c(\emptyset)\neq c([n])$.
If there is no rainbow induced copy of $\mathcal{B}_2$, 
then $X\sim Y$ for any sets $X$ and $Y$ with $c(X), c(Y)\notin \{c(\emptyset), c([n])\}$ and $c(X)\neq c(Y)$.  
\end{lemma}
\begin{proof}
If there exist two sets $X_0, Y_0$ with $c(X_0), c(Y_0)\notin \{c(\emptyset), c([n])\}$ and $c(X_0)\neq c(Y_0)$ such that $X_0\nsim Y_0$, then there exists a rainbow induced copy of $\mathcal{B}_2$ with family $\{\emptyset, X_0, Y_0, [n]\}$, a contradiction. Therefore, 
the conclusion holds.
\end{proof}

Now, we are ready to prove Lemma \ref{k5}.

\begin{proof}[Proof of Lemma \ref{k5}]
We first prove the ``if'' part.
Suppose that the coloring $c$ is of Type 1.
Without loss of generality, we may assume
$c(Z)=1$ for each $Z\in \mathcal{B}_{[\emptyset,Y_0]}\setminus\mathcal{B}_{[X_0,Y_0]}$,
$c(X_0)=2$,
$c(Z)=3$ for each $Z\in \mathcal{B}_{(X_0,Y_0)}$,
$c(Y_0)=4$,
$c(Z)=5$ for each $Z\in \mathcal{B}_{[X_0,[n]]}\setminus\mathcal{B}_{[X_0,Y_0]}$, and every set in $\mathcal{B}_{n} \setminus \left(\mathcal{B}_{[\emptyset, Y_0]}\cup \mathcal{B}_{[X_0, [n]]}\right)$ is colored $1$ or $5$.
Note that for each $Z \in \Bn$, if $c(Z) = 1$ then $Z \notin \mathcal{B}_{[X_0, [n]]}$, if $c(Z) = 5$ then
$Z \notin \mathcal{B}_{[\emptyset, Y_0]}$, and if $c(Z)\in \{2,3,4\}$ then
$Z \in \mathcal{B}_{[X_0, Y_0]}$.

Assume for contradiction that there exists a rainbow induced copy of $\mathcal{B}_2$
by the sets $W_1, W_2, W_3, W_4 \in \Bn$ with $W_1<W_2, W_3 <W_4$ and $W_2\nsim W_3$. 
Suppose first that $c(W_1)=1$. 
Since only $X_0$, the sets in $\mathcal{B}_{(X_0, Y_0)}$, and $Y_0$ are colored $2$, $3$, and $4$, respectively, and $W_2\nsim W_3$, one of $W_2$ and $W_3$ is colored $5$. Without loss of generality, assume $c(W_2) = 5$.
This means $W_2\in \mathcal{B}_n\setminus \mathcal{B}_{[\emptyset, Y_0]}$. 
Since $W_4 > W_2$ must satisfy $W_4 \in \mathcal{B}_n\setminus \mathcal{B}_{[\emptyset, Y_0]}$, 
which implies $c(W_4) \in \{1, 5\}$, 
a contradiction.

Suppose next that $c(W_1)=2$.
Then $W_1=X_0$ and $W_2, W_3, W_4\in \mathcal{B}_{(X_0, [n]]}$.
By the conditions on a coloring of Type 1, 
we have $\{c(W_2), c(W_3), c(W_4)\}= \{3, 4, 5\}$.
Since
the colors $3$ and $4$ appear only on the sets in $\mathcal{B}_{(X_0, Y_0)}$ and on $Y_0$, respectively, 
we have $W_4=Y_0$.
This implies $W_2, W_3 \in \mathcal{B}_{(X_0, Y_0)}$, 
and hence $c(W_2) = c(W_3) = 3$, a contradiction.

If $c(W_1)\in \{3, 4\}$, then $W_1\in \mathcal{B}_{(X_0, Y_0]}$, and hence $c(Z)\in \{3, 4, 5\}$ for each $Z\in \mathcal{B}_{(W_1, [n]]}$, and so $\{c(W_2), c(W_3), c(W_4)\}\subseteq \{3, 4, 5\}$, a contradition.
If $c(W_1)=5$, then there exists no set $Z$ with $Z>W_1$ such that $Z\in \mathcal{B}_{[X_0,Y_0]}$, and hence $c(Z)\in \{1, 5\}$ for each $Z\in \mathcal{B}_{(W_1, [n]]}$, and so $\{c(W_2), c(W_3),$ $c(W_4)\}\subseteq \{1, 5\}$, a contradition.
Thus, no rainbow induced copy of $\mathcal{B}_2$ exists.
\medskip

Next, we prove the ``only if'' part. Suppose that there is no rainbow induced copy of $\mathcal{B}_2$.
Since $k\geq 5$, it follows from Lemma \ref{lem-x-y} that there exists a rainbow $5$-chain $(\emptyset, Z_2, Z_3, Z_4, [n])$.
We may assume that $c(\emptyset) = 1$, $c(Z_i) = i$ for $i \in \{2, 3, 4\}$, and $c([n]) = 5$.
We will show that $c$ is of Type 1 with $X_0 = Z_2$ and $Y_0 = Z_4$. We first show two common claims for the colors of some sets.

\begin{claim}\label{cl-21}
$c(Y)=3$ for each $Y\in \mathcal{B}_{(Z_2, Z_4)}$.
\end{claim}
\begin{proof}
Let $Y\in\mathcal{B}_{(Z_2, Z_4)}$. 
Suppose that $Y\nsim Z_3$. 
If $c(Y)\in\{2, 4\}$, then $\{\emptyset, Y, \, Z_3, \, [n]\}$ forms a rainbow induced copy of $\mathcal{B}_2$; if $c(Y)\in\{1, 5\}$ then $\{Z_2, \, Y, \, Z_3, \, Z_4\}$ forms a rainbow induced copy of $\mathcal{B}_2$, a contradiction. Therefore, $c(Y)=3$.

Suppose that $Y\sim Z_3$. Then 
either $Z_2< Y< Z_3$ or $Z_3< Y< Z_4$.
If $Z_2< Y< Z_3$, then pick $a\in Y\setminus Z_2$ and $b\in Z_4\setminus Z_3$, and define $Y'=(Z_3\setminus\{a\})\cup\{b\}$.
Therefore, $Y'\in \mathcal{B}_{(Z_2, Z_4)}$, $a\in Z_3\setminus Y'$, and $b\in Y'\setminus Z_3$, and hence $Y'\nsim Z_3$; also $a\in Y\setminus Y'$ and $b\in Y'\setminus Y$, and thus $Y'\nsim Y$.
If $Z_3< Y< Z_4$, pick $a\in Z_3\setminus Z_2$, $b\in Z_4\setminus Y$, and define $Y'$ as above. The same witnesses show $Y'\nsim Z_3$ and $Y'\nsim Y$.
By the previous paragraph, we have $c(Y')=3$.
If $c(Y)\in\{2, 4\}$, then $\{\emptyset, Y, Y', [n]\}$ forms a rainbow induced copy of $\mathcal{B}_2$; if $c(Y)\in\{1, 5\}$ then $\{Z_2, Y, Y', Z_4\}$ forms a rainbow induced copy of $\mathcal{B}_2$, a contradiction. Therefore, $c(Y)=3$.
\end{proof}

\begin{claim}\label{cl-22}
$c(Y)=1$ for each $Y\in \mathcal{B}_{[\emptyset, Z_4]}\setminus \mathcal{B}_{[Z_2, Z_4]}$.
\end{claim}
\begin{proof}
Let $Y\in \mathcal{B}_{[\emptyset, Z_4]}\setminus \mathcal{B}_{[Z_2, Z_4]}$.
Suppose first that $Y \notin \mathcal{B}_{[\emptyset, Z_2)}$. Then $Y\nsim Z_2$. By Lemma \ref{lem:TL1} with $X=Z_2$, $Z=Z_4$, and $W=Y$, there exists a set $Y'\in\mathcal{B}_{(Z_2, Z_4)}$ such that $Y\nsim Y'$. By Claim \ref{cl-21}, we have $c(Y') = 3$.
If $c(Y)\in\{2, 4\}$, then $\{\emptyset, Y, Y', [n]\}$ forms a rainbow induced copy of $\mathcal{B}_2$; if $c(Y)\in\{3, 5\}$, then $\{\emptyset, Z_2, Y, Z_4\}$ forms a rainbow induced copy of $\mathcal{B}_2$, a contradiction. Therefore, $c(Y)=1$.

Next suppose that $Y \in \mathcal{B}_{[\emptyset, Z_2)}$. Since $c(\emptyset)=1$, it remains to treat $Y\in\mathcal{B}_{(\emptyset, Z_2)}$. Choose $w\in Z_4\setminus Z_2$ and $u\in (Z_4\setminus Z_2)\setminus\{w\}$. Then let $Y'= Z_2\cup\{w\}$ and $Y''= Y\cup\{u\}.$
Then $Y'\in \mathcal{B}_{(Z_2, Z_4)}$ and $Y''\in \mathcal{B}_{(Y, Z_4)}\setminus\mathcal{B}_{[\emptyset, Z_2]}$. By Claim \ref{cl-21}, we know $c(Y')=3$.
Choose $v\in Z_2\setminus Y$. Then $v\notin Y$ and hence $v\notin Y''$.
Since $Z_2\subseteq Y'$ we have $v\in Y'$, and thus $v\in Y'\setminus Y''$. Therefore, $Y'\not\leq Y''$.
Since $u\notin Z_2$, we have $u\in Y''\setminus Z_2$. Note that $w\neq u$. Thus, $u\notin Y'$, and hence $u\in Y''\setminus Y'$, giving $Y''\not\leq Y'$.
Therefore, $Y'\nsim Y''$.
Moreover, $u\in Y''\setminus Z_2$ shows $Y''\not\leq Z_2$, while $v\in Z_2\setminus Y''$ shows $Z_2\not\leq Y''$, and so $Y''\nsim Z_2$. 
Note that $Y''\notin\mathcal{B}_{[\emptyset, Z_2]}$. By
the previous paragraph, $c(Y'')=1$.
If $c(Y)\in\{2, 4\}$, then $\{Y, Y', Y'', [n]\}$ forms a rainbow induced copy of $\mathcal{B}_2$; if $c(Y)\in\{3, 5\}$, then $\{Y, Z_2, Y'', Z_4\}$ forms a rainbow induced copy of $\mathcal{B}_2$, a contradiction.
Therefore, $c(Y)=1$.
\end{proof}

By symmetry, interchange the roles of $\emptyset$ and $[n]$ and also the roles of $Z_2$ and $Z_4$, and at the same time relabel the colors by replacing every occurrence of color $1$ with color $5$ and vice versa, and replacing color $2$ with color $4$ and vice versa, while leaving color $3$ unchanged. The same argument as in Claim \ref{cl-22} shows that $c(Y)=5$ for every $Y\in \mathcal{B}_{[Z_2, [n]]}\setminus \mathcal{B}_{[Z_2, Z_4]}$.

Let $Y\in \mathcal{B}_n\setminus (\mathcal{B}_{[Z_2, [n]]}\cup \mathcal{B}_{[\emptyset, Z_4]})$. Then $Y\nsim Z_2$ and $Y\nsim Z_4$. If $c(Y)\in\{2, 3\}$, then $\{\emptyset, Z_4, Y, [n]\}$ forms a rainbow induced copy of $\mathcal{B}_2$; if $c(Y)=4$, then $\{\emptyset, Z_2, Y, [n]\}$ forms a rainbow induced copy of $\mathcal{B}_2$, a contradiction.
Thus, 
$c(Y)\in \{1, 5\}$.

Therefore, we conclude that the coloring $c$ is of Type 1 with $X_0 = Z_2$ and $Y_0 = Z_4$, 
which completes the proof of the ``only if'' part.
\end{proof}

\subsection{An exact $4$-coloring $c$ without a rainbow induced copy of $\mathcal{B}_2$ with $c(\emptyset)\neq c([n])$}
\label{subsection24}

{We now introduce several types of exact \(4\)-colorings of \(\mathcal{B}_n\), each obtained from the coloring of Type~1 by making the colors on two of the following families the same:
\(\mathcal{B}_{[\emptyset, Y_0]} \setminus \mathcal{B}_{[X_0, Y_0]}\), \(\{X_0\}\), \(\mathcal{B}_{(X_0, Y_0)}\), \(\{Y_0\}\), and \(\mathcal{B}_{[X_0,[n]]} \setminus \mathcal{B}_{[X_0, Y_0]}\).
Since \(c(\emptyset)\neq c([n])\), we do not make the colors on
\(\mathcal{B}_{[\emptyset, Y_0]} \setminus \mathcal{B}_{[X_0, Y_0]}\) and
\(\mathcal{B}_{[X_0,[n]]} \setminus \mathcal{B}_{[X_0, Y_0]}\) the same.

We define the resulting types as follows.
Type~2 is obtained by making the colors on \(\{X_0\}\) and \(\{Y_0\}\) the same.
Type~3\textnormal{-}1 is obtained by making the colors on \(\{X_0\}\) and \(\mathcal{B}_{[X_0,[n]]} \setminus \mathcal{B}_{[X_0, Y_0]}\) the same.
Type~3\textnormal{-}2 is obtained by making the colors on \(\mathcal{B}_{[\emptyset, Y_0]} \setminus \mathcal{B}_{[X_0, Y_0]}\) and \(\{Y_0\}\) the same.
Type~4\textnormal{-}1 is obtained by making the colors on two of
\(\mathcal{B}_{(X_0, Y_0)}\), \(\{Y_0\}\), and \(\mathcal{B}_{[X_0,[n]]} \setminus \mathcal{B}_{[X_0, Y_0]}\) the same.
Type~4\textnormal{-}2 is obtained by making the colors on two of
\(\mathcal{B}_{[\emptyset, Y_0]} \setminus \mathcal{B}_{[X_0, Y_0]}\), \(\{X_0\}\), and \(\mathcal{B}_{(X_0, Y_0)}\) the same.

In fact, only Type~2 can be obtained directly from  Type~1 by making \(\{X_0\}\) and \(\{Y_0\}\) the same color. For the other types, however, making the indicated two families the same color forces additional changes elsewhere in the coloring, so the coloring on the remaining parts of \(\mathcal{B}_n\) is not identical to that of Type~1.}

We then prove that these are in fact the only exact $4$-colorings of $\Bn$ with $c(\emptyset)\neq c([n])$ such that  there is no rainbow induced copy of $\BB$.

{Based on Type~1, we define Type~2 by making the colors on \(\{X_0\}\) and \(\{Y_0\}\) the same.} 
\begin{itemize}
\item
\emph{Type $2$}: 
There exist two sets $X_0, Y_0 \in \Bn$ with $\emptyset < X_0 < Y_0 < [n]$ and $|Y_0| \geq |X_0| + 2$
such that all sets in the following families are monochromatically colored
so that no two sets from different families share the same color:
$$\mathcal{B}_{[\emptyset, Y_0]} \setminus \mathcal{B}_{[X_0, Y_0]}, \quad \{X_0, Y_0\}, \quad \mathcal{B}_{(X_0, Y_0)}, \quad \mathcal{B}_{[X_0, [n]]} \setminus \mathcal{B}_{[X_0, Y_0]}.$$

Moreover, 
each set in 
$\Bn \setminus \left(\mathcal{B}_{[\emptyset, Y_0]} \cup \mathcal{B}_{[X_0, [n]]}\right)$
has the same color as
the sets in the first or the last family.
\end{itemize}
Note that both $\mathcal{B}_{[\emptyset, Y_0]}$ and $\mathcal{B}_{[X_0, [n]]}$ are colored with exactly $3$ colors.

{Based on Type~1, we define Type~3\textnormal{-}1 by making the colors on \(\{X_0\}\) and \(\mathcal{B}_{[X_0,[n]]}\setminus \mathcal{B}_{[X_0,Y_0]}\) the same.
Compared with Type~1, the colors of all sets in
$\mathcal{B}_n\setminus\bigl(\mathcal{B}_{[\emptyset,Y_0]}\setminus\{X_0\}\bigr)$ are changed, while the colors on $\mathcal{B}_{[\emptyset,Y_0]}\setminus\{X_0\}$ remain the same. 

Based on Type~1, we define Type~3-2 by making the colors on \(\mathcal{B}_{[\emptyset, Y_0]} \setminus \mathcal{B}_{[X_0, Y_0]}\) and \(\{Y_0\}\) the same.
Compared with Type~1, the colors of all sets in
$\mathcal{B}_n\setminus\bigl(\mathcal{B}_{[X_0,[n]]}\setminus\{Y_0\}\bigr)$ are changed, while the colors on $\mathcal{B}_{[X_0,[n]]}\setminus\{Y_0\}$ remain the same. 
We first introduce the following notation.}

Let $X_0, Y_0 \in \mathcal{B}_n$ be two sets with $X_0<Y_0$. Before introducing Types~3-1 and~3-2, we define a subfamily $\mathcal{B}^{Y_0\uparrow}_{(X_0, [n]]}$ of $\mathcal{B}_{(X_0, [n]]}$ by
$$
\mathcal{B}^{Y_0\uparrow}_{(X_0, [n]]} = \bigcup_{X_0' \in \mathcal{B}_{(X_0, Y_0)}} \mathcal{B}_{[X_0', [n]]}, 
$$
and a subfamily $\mathcal{B}^{X_0\downarrow}_{[\emptyset, Y_0)}$ of $\mathcal{B}_{[\emptyset, Y_0)}$ by
$$
\mathcal{B}^{X_0\downarrow}_{[\emptyset, Y_0)} = \bigcup_{Y_0' \in \mathcal{B}_{(X_0, Y_0)}} \mathcal{B}_{[\emptyset, Y_0']}.
$$
Equivalently, 
$
\mathcal{B}^{Y_0\uparrow}_{(X_0, [n]]}=\{X\in \mathcal{B}_{(X_0, [n]]} :  X\cap (Y_0\setminus X_0)\neq \emptyset\} 
$
and 
$\mathcal{B}^{X_0\downarrow}_{[\emptyset, Y_0)}=\{Y\in \mathcal{B}_{[\emptyset, Y_0)} :  Y_0\setminus X_0\not\subseteq Y\}.$
Note that $\mathcal{B}_{(X_0, Y_0)}
=\;\mathcal{B}^{Y_0\uparrow}_{(X_0, [n]]}\;\cap\;\mathcal{B}^{X_0\downarrow}_{[\emptyset, Y_0)}$.
The “up” family $\mathcal{B}^{Y_0\uparrow}_{(X_0, [n]]}$ consists of sets above $X_0$ that meet $Y_0\setminus X_0$; the “down” family $\mathcal{B}^{X_0\downarrow}_{[\emptyset, Y_0)}$ consists of sets below $Y_0$ that do not contain all of $Y_0\setminus X_0$. If we swap the bottom and the top and interchange $X_0$ and $Y_0$, “meet $Y_0\setminus X_0$ above $X_0$” becomes “miss $Y_0\setminus X_0$ below $Y_0$”. Thus the two families are symmetric.

\begin{itemize}
\item
\emph{Type $3$-$1$}:
There exist two sets $X_0, Y_0 \in \Bn$ with $\emptyset < X_0 < Y_0 < [n]$ and $|Y_0| \geq |X_0| + 2$
such that
all sets in the following families are monochromatically colored
so that no two sets from different families share the same color:
$$
\mathcal{B}_{[\emptyset, Y_0]} \setminus \mathcal{B}_{[X_0, Y_0]},  \quad
\{X_0\}\cup \mathcal{B}^{Y_0\uparrow}_{(X_0, [n]]}\setminus \mathcal{B}_{(X_0, Y_0]}, \quad 
\mathcal{B}_{(X_0, Y_0)}, \quad
\{Y_0\}.
$$
Moreover, 
each set in $\Bn \setminus \left(\mathcal{B}_{[\emptyset, Y_0]} \cup \mathcal{B}^{Y_0\uparrow}_{(X_0, [n]]}\right)$
has the same color as
the sets in the first or the second family.
\end{itemize}
Note that both $\mathcal{B}_{[\emptyset, Y_0]}\setminus\{X_0\}$ and $\mathcal{B}^{Y_0\uparrow}_{(X_0, [n]]}$ are colored by 3 colors.

\begin{itemize}
\item
\emph{Type $3$-$2$}:
There exist two sets $X_0, Y_0 \in \Bn$ with $\emptyset < X_0 < Y_0 < [n]$ and $|Y_0| \geq |X_0| + 2$
such that all sets in the following families are monochromatically colored
so that no two sets from different families share the same color:
$$
\mathcal{B}_{[X_0, [n]]} \setminus \mathcal{B}_{[X_0, Y_0]}, \quad
\{Y_0\}\cup\mathcal{B}^{X_0\downarrow}_{[\emptyset, Y_0)}\setminus \mathcal{B}_{[X_0, Y_0)}, \quad
\mathcal{B}_{(X_0, Y_0)}, \quad
\{X_0\}.
$$
Moreover, 
each set in $\Bn \setminus \left(\mathcal{B}_{[X_0, [n]]} \cup \mathcal{B}^{X_0\downarrow}_{[\emptyset, Y_0)}\right)$
has the same color as
the sets in the first or the second family.
\end{itemize}
%%%%%%%%%%%%%%%%%%%%%%%%%%%%%%%%%%%%%%%%%%%%%%%%%%%%%%%%%%%%%%%%%%%%%%%%%%%%%%%%
{Based on Type~1, we define Type~4\textnormal{-}1 by making the colors on two of
\(\mathcal{B}_{(X_0, Y_0)}\), \(\{Y_0\}\), and \(\mathcal{B}_{[X_0,[n]]}\setminus \mathcal{B}_{[X_0, Y_0]}\) the same.
If \(\mathcal{Y}_0=\mathcal{B}_{(X_0, Y_0]}\), then the colors on \(\mathcal{B}_{(X_0, Y_0)}\) and \(\{Y_0\}\) are the same.
If \(\mathcal{Y}_0=\{Y_0\}\), then the colors on \(\mathcal{B}_{(X_0, Y_0)}\) and \(\mathcal{B}_{[X_0,[n]]}\setminus \mathcal{B}_{[X_0, Y_0]}\) are the same.
If \(\mathcal{Y}_0=\mathcal{B}_{(X_0, Y_0)}\), then the colors on \(\{Y_0\}\) and \(\mathcal{B}_{[X_0,[n]]}\setminus \mathcal{B}_{[X_0, Y_0]}\) are the same.}

\begin{itemize}

\item
\emph{Type $4$-$1$}:
There exists a set $X_0$ with $1 \leq |X_0| \leq n-2$
and a family $\mathcal{Y}_0 \subseteq \mathcal{B}_{(X_0, [n]]}$ with $\mathcal{Y}_0 \neq \emptyset$
such that 
all sets in the following families are monochromatically colored
so that no two sets from different families share the same color:
$$
\bigcup_{Y \in \mathcal{Y}_0} \left(\mathcal{B}_{[\emptyset, Y)} \setminus \mathcal{B}_{[X_0, Y)}\right), \quad 
\{X_0\}, \quad 
\mathcal{B}_{(X_0, [n]]} \setminus \mathcal{Y}_0, \quad
\mathcal{Y}_0.
$$
Moreover, each set in
$\Bn \setminus \left(\mathcal{B}_{[X_0, [n]]} \cup \bigcup_{Y\in \mathcal{Y}_0}\mathcal{B}_{[\emptyset, Y)}\right)$
has the same color as
the sets in the first or the third family.
\end{itemize}
Based on Type~1, we define Type~4-2 by merging two of colors appearing on $
\mathcal{B}_{[\emptyset, Y_0]} \setminus \mathcal{B}_{[X_0, Y_0]}, \{X_0\}$, and  $\mathcal{B}_{(X_0, Y_0)}$.
If we merge two of colors appearing on $ \mathcal{B}_{[\emptyset, Y_0]} \setminus \mathcal{B}_{[X_0, Y_0]}$ and $\{X_0\}$, 
then let $\mathcal{X}_0=\mathcal{B}_{(X_0, Y_0)}$; if we merge two of colors appearing on $ \mathcal{B}_{[\emptyset, Y_0]} \setminus \mathcal{B}_{[X_0, Y_0]}$ and $\mathcal{B}_{(X_0, Y_0)}$, then let 
$\mathcal{X}_0=\{X_0\}$; if we merge two of colors appearing on $\{X_0\}$ and $\mathcal{B}_{(X_0, Y_0)}$, then let 
$\mathcal{X}_0=\mathcal{B}_{[X_0, Y_0)}$.

\begin{itemize}
\item
%[$(3.2)$]  
\emph{Type $4$-$2$}:
There exists a set $Y_0$ with $2 \leq |Y_0| \leq n-1$
and a family $\mathcal{X}_0 \subseteq \mathcal{B}_{(\emptyset, Y_0)}$ with $\mathcal{X}_0\neq \emptyset$
such that 
all sets in the following families are monochromatically colored
so that no two sets from different families share the same color:
$$
\bigcup_{X \in \mathcal{X}_0} \left(\mathcal{B}_{(X, [n]]} \setminus \mathcal{B}_{(X, Y_0]}\right), \quad 
\{Y_0\}, \quad \mathcal{B}_{[\emptyset, Y_0)} \setminus \mathcal{X}_0, \quad 
\mathcal{X}_0.
$$

%are monochromatic with distinct colors. 
Moreover, each set in
$\Bn \setminus \left(\mathcal{B}_{[\emptyset, Y_0]}\cup\bigcup_{X\in \mathcal{X}_0} \mathcal{B}_{(X, [n]]}\right)$
has the same color as
the sets in the first or the third family.
\end{itemize}
Swapping $X_0$ with $Y_0$ and $\mathcal{Y}_0$ with $\mathcal{X}_0$, and reversing inclusion, 
Type~4-1 is transformed into Type~4-2 and vice versa. Thus, the two types are symmetric.

The following observation will be useful in the proof of Lemma~\ref{emptyset=n}.
\begin{observation}\label{obs1}
In Types~$2$, $3$-$1$, $3$-$2$, and $4$-$2$, there exists no set $Z$ with $\mathcal{B}_{(Z, [n])}\neq \emptyset$ such that 
all sets in the families $\{Z\}$, 
$\mathcal{B}_{(Z, [n])}$, $\{[n]\}$ are monochromatically colored, and moreover no two sets from different families share
the same color.  
\end{observation}

\begin{lemma}\label{k4}
Let $n$ be an integer with $n \geq 4$.
Consider an exact $4$-coloring $c$ of $\mathcal{B}_n$
with $c(\emptyset)\neq c([n])$.
Then there is no rainbow induced copy of $\mathcal{B}_2$
if and only if the coloring $c$ is of Type $2$, $3$-$1$, $3$-$2$, $4$-$1$ or $4$-$2$.
\end{lemma}
\begin{proof}
We first prove the ``if'' part.
Note that Type 3-1 is symmetric to Type 3-2, 
and Type 4-1 is symmetric to Type 4-2.
%the pairs (Type 3, Type 4) and (Type 5, Type 6) are symmetric, respectively.
Hence, it suffices to show that 
there is no rainbow induced copy of $\mathcal{B}_2$ for 
an exact $4$-coloring $c$ of Type 2, Type 3-1, and Type 4-1. 
%Conversely, suppose that one of Type 2, Type 3, ..., Type 6 holds.
Assume for contradiction that there exists a rainbow induced copy of $\mathcal{B}_2$
by the sets $W_1, W_2, W_3, W_4 \in \Bn$ with $W_1<W_2, W_3 <W_4$ and $W_2\nsim W_3$.

\begin{itemize}
\item
Suppose that $c$ is of Type 2.
Since Type~2 can be obtained directly from Type~1 by assigning the same color to $X_0$ and $Y_0$, 
there is no rainbow induced copy of $\mathcal{B}_2$
for an exact $4$-coloring of Type $2$.

\item
Suppose next that $c$ is of Type $3$-$1$.
Without loss of generality, 
we may assume that each of 
$\mathcal{B}_{[\emptyset, Y_0]} \setminus \mathcal{B}_{[X_0, Y_0]}$, 
$\{X_0\}\cup\mathcal{B}^{Y_0\uparrow}_{(X_0, [n]]}\setminus \mathcal{B}_{(X_0, Y_0]}$, 
$\mathcal{B}_{(X_0, Y_0)}$, and $\{Y_0\}$ is monochromatic 
with colors $1, 2, 3, 4$, respectively, 
which implies that 
each set in $\Bn \setminus \left(\mathcal{B}_{[\emptyset, Y_0]} \cup \mathcal{B}^{Y_0\uparrow}_{(X_0, [n]]}\right)$
has color either $1$ or $2$.

Note that $Y_0$ is the only set of color $4$, 
and any set of color $3$ is contained in $\mathcal{B}_{(X_0, Y_0)}$, 
which implies $c(W_1) = 3$ or $c(W_4) = 4$ (or both).
If $c(W_1) = 3$, 
that is, if $W_1 \in \mathcal{B}_{(X_0, Y_0)}$, 
then 
each set $Z \in \mathcal{B}_{[W_1, [n]]}$ 
is contained in $\mathcal{B}^{Y_0\uparrow}_{(X_0, [n]]}$, 
and hence $c(Z) = 2$ (when $Z \notin \mathcal{B}_{(X_0, Y_0]}$), 
$c(Z) = 3$ (when $Z \in \mathcal{B}_{(X_0, Y_0)}$), 
or
$c(Z) = 4$ (when $Z = Y_0$), 
contradicting that one of $W_2, W_3$ and $W_4$ has color $1$.

Therefore, we may assume that $c(W_1) \neq 3$ and $c(W_4) = 4$, 
that is, $W_4 = Y_0$.
By symmetry between $W_2$ and $W_3$, we may assume that $c(W_2)=3$, and hence $W_2 \in \mathcal{B}_{(X_0, Y_0)}$.
Since $X_0$ is the only set in $\mathcal{B}_{[\emptyset, Y_0)}$ of color $2$
and $X_0 < W_2$, 
we have $W_1 = X_0$.
However, 
all sets in $\mathcal{B}_{(W_1, W_4)} = \mathcal{B}_{(X_0, Y_0)}$ have the color $3$, 
a contradiction.

\item
Suppose that $c$ is of Type $4$-$1$.
Without loss of generality, assume that each of  
$\bigcup_{Y \in \mathcal{Y}_0} \left(\mathcal{B}_{[\emptyset, Y)} \setminus \mathcal{B}_{[X_0, Y)}\right)$, $\{X_0\}$, $\mathcal{B}_{(X_0, [n]]} \setminus \mathcal{Y}_0$, and $\mathcal{Y}_0$
is monochromatic with colors $1, 2, 3, 4$, respectively, 
which implies that 
%all remaining sets are colored 
each set in $\Bn \setminus \left(\mathcal{B}_{[X_0, [n]]} \cup \bigcup_{Y\in \mathcal{Y}_0}\mathcal{B}_{[\emptyset, Y)}\right)$
is colored either $1$ or $3$.

Note that $X_0$ is the only set of color $2$, 
and any set of color $4$ is contained in $\mathcal{Y}_0$, 
which implies $c(W_1) = 2$ or $c(W_4) = 4$ (or both).
If $c(W_1) = 2$, 
that is, if $W_1 = X_0$, 
then 
for each set $Z \in \mathcal{B}_{[W_1, [n]]}$, 
we have $c(Z) = 3$ (when $Z \notin \mathcal{Y}_0$), 
or
$c(Z) = 4$ (when $Z \in \mathcal{Y}_0$), 
contradicting that one of $W_2, W_3$ and $W_4$ has color $1$.

Therefore, we may assume that $c(W_1) \neq 2$ and $c(W_4) = 4$, 
that is, $W_4 \in \mathcal{Y}_0$.
By symmetry between $W_2$ and $W_3$, we may assume that $c(W_2) = 2$, 
and hence $W_2 = X_0$.
Since any set $Z \in \mathcal{B}_{[\emptyset, X_0)}$ satisfies $c(Z) = 1$, 
we see $c(W_1) = 1$, 
and hence $c(W_3) = 3$.
Since any set $Z$ in $\mathcal{B}_{(W_1, W_4)}$ with $c(Z) = 3$
has to be contained in $\mathcal{B}_{(X_0, [n]]} \setminus \mathcal{Y}_0$, 
we see $W_2 < W_3$, 
a contradiction.
\end{itemize}

This completes the proof of  the ``if'' part.
\medskip

We next prove the ``only if'' part.
Consider an exact $4$-coloring $c$ of $\mathcal{B}_n$
with $c(\emptyset)\neq c([n])$.
We may assume that $c(\emptyset) = 1$ and $c([n]) = 4$.
Since $c$ is an exact $4$-coloring, 
there exist two sets $X_2, X_3 \in \Bn$ with 
$c(X_2) = 2$ and $c(X_3) = 3$.
By Lemma \ref{lem-x-y}, 
we may assume $X_2 < X_3$.
By taking such two sets $X_2$ and $X_3$
so that $|X_3| - |X_2|$ is as large as possible, 
we can assume that $c(X) \neq 2$ for each $X$ with $|X|<|X_2|$
and $c(X) \neq 3$ for each $X$ with $|X|>|X_3|$.

Therefore, one of the following cases holds.
Case 1: There exists $X_0 \in \mathcal{B}_{(\emptyset, X_2)}$ such that $c(X_0) \in \{3, 4\}$; Case 2: There exists $Y_0 \in \mathcal{B}_{(X_3, [n])}$ with $c(Y_0)\in\{1, 2\}$; Case 3: All sets in $\mathcal{B}_{[\emptyset, X_2)}$ are colored $1$ and all sets in $\mathcal{B}_{(X_3, [n]]}$ are colored $4$. 
These three possibilities cover all colorings, so we analyze them separately.

\begin{case}\label{case1}
There exists a set $X_0 \in \mathcal{B}_{(\emptyset, X_2)}$ such that $c(X_0) \in \{3, 4\}$. 
\end{case}

We will later divide this case further into two subcases depending on $c(X_0)$, 
and show that $c$ is of Type 2 or $3$-$1$, 
respectively, with $X_0$ and $Y_0 = X_3$.
Since $X_0 < X_2 < X_3$, we have $|X_3| \geq |X_0| + 2$.

\begin{claim}
\label{Case1claim1}
$c(Y)= 2$ for each $Y\in \mathcal{B}_{(X_0, X_3)}$.
\end{claim}

\begin{proof}
Let $Y\in \mathcal{B}_{(X_0, X_3)}$.
Suppose first that $Y \nsim X_2$. 
If $c(Y)=1$, then there is a rainbow induced copy of $\mathcal{B}_2$ by $\{X_0, X_2, Y, [n]\}$ (when $c(X_0)=3$) or $\{X_0, X_2, Y, X_3\}$ (when $c(X_0)=4$);
if $c(Y)= 3$, then there is a rainbow induced copy of $\mathcal{B}_2$ by $\{\emptyset, X_2, Y, [n]\}$;
if $c(Y)= 4$, then there is a rainbow induced copy of $\mathcal{B}_2$ by $\{\emptyset, X_2, Y, X_3\}$, 
a contradiction.
Therefore, $c(Y) = 2$.

We next suppose that $Y \sim X_2$. 
In this case, 
we can find a set $Y_1\in \mathcal{B}_{(X_0, X_3)}$ with $Y_1\nsim Y$ and $Y_1 \nsim X_2$
(e.g., if $Y<X_2$, then take $Y_1=(X_2\cup\{a\})\setminus \{b\}$, where $a\in X_3\setminus X_2$ and $b\in Y$; if $Y>X_2$, then take $Y_1=(X_2\cup\{a\})\setminus \{b\}$, where $a\in X_3\setminus Y$ and $b\in X_2$).
By the fact obtained in the previous paragraph, we see $c(Y_1) = c(X_2) = 2$, 
and by the same argument as above by replacing $X_2$ with $Y_1$, 
we obtain $c(Y) = c(Y_1) = 2$.
This completes the proof of Claim \ref{Case1claim1}.
\end{proof}

\begin{claim}
\label{Case1claim2}
$c(Y)= 1$ for each $Y\in \mathcal{B}_{[\emptyset, X_3]}\setminus \mathcal{B}_{[X_0, X_3]}$.
\end{claim}

\begin{proof}
Let $Y\in \mathcal{B}_{[\emptyset, X_3]}\setminus \mathcal{B}_{[X_0, X_3]}$.
Suppose first that $Y \notin \mathcal{B}_{[\emptyset, X_0)}$. Then $Y\nsim X_0$.
Note that $Y \notin \mathcal{B}_{[X_3, [n]]}$.
By Lemma \ref{lem:TL1} with $X = X_0$, $Z = X_3$,  and $W=Y$, there exists a set $W_0 \in \mathcal{B}_{(X_0, X_3)}$ such that $W_0\nsim Y$. 
By Claim \ref{Case1claim1}, $c(W_0)=2$.
If $c(Y)=2$, then there is a rainbow induced copy of $\mathcal{B}_2$ by $\{\emptyset, Y, X_0, [n]\}$
(when $c(X_0)=3$) or $\{\emptyset, Y, X_0, X_3\}$ (when $c(X_0)=4$);
if $c(Y)=3$, then there is a rainbow induced copy of $\mathcal{B}_2$ by $\{\emptyset, Y, W_0, [n]\}$; 
if $c(Y)=4$, then there is a rainbow induced copy of $\mathcal{B}_2$ by $\{\emptyset, Y, W_0, X_3\}$, a contradiction. 
Therefore, $c(Y)= 1$.

We next suppose that $Y \in \mathcal{B}_{[\emptyset, X_0)}$, 
that is, $\emptyset \leq Y < X_0 < X_2$. Since $c(\emptyset) = 1$, we may assume $Y \neq \emptyset$.
By the choice of $X_2$, we see $c(Y) \neq 2$.
Let $a\in X_0\setminus Y$ and $b\in X_3\setminus X_0$. Then, 
$X_3\setminus \{a\} \in \mathcal{B}_{(Y, X_3)}\setminus \mathcal{B}_{[X_0, X_3)}$
and $X_3\setminus \{b\}\in \mathcal{B}_{(X_0, X_3)}$, 
Note that $X_3\setminus \{a\}\nsim X_3\setminus \{b\}$.
By the fact obtained in the previous paragraph, $c(X_3\setminus \{a\})=1$. 
By Claim \ref{Case1claim1}, $c(X_3\setminus \{b\})=2$.
If $c(Y) = 3$ or $c(Y) = 4$, 
then there is a rainbow induced copy of $\mathcal{B}_2$
by $\{Y, X_3\setminus \{a\}, X_3\setminus \{b\}, [n]\}$ or $\{Y, X_3\setminus \{a\}, X_3\setminus \{b\}, X_3\}$, respectively, a contradiction.
Therefore, 
$c(Y)= 1$, 
and this completes the proof of Claim \ref{Case1claim2}.
\end{proof}

By Claims \ref{Case1claim1} and \ref{Case1claim2}, we show the colors of the sets in $\mathcal{B}_{[\emptyset, X_3)}\setminus \{X_0\}$.
We now consider the colors of the sets in $\mathcal{B}_n\setminus \mathcal{B}_{[\emptyset, X_3)}$ under two subcases, 
based on the color of $X_0$.

\begin{subcase}
$c(X_0) =3$.    
\end{subcase}

In this subcase, we will show that $c$ is of Type 2 with $X_0$ and $Y_0 = X_3$.

We now prove that $c(Y)=4$ for each
$Y\in \mathcal{B}_{[X_0, [n]]}\setminus \mathcal{B}_{[X_0, X_3]}$.
Let $Y\in \mathcal{B}_{[X_0, [n]]}\setminus \mathcal{B}_{[X_0, X_3]}$. 
Suppose that $Y \notin \mathcal{B}_{(X_3, [n]]}$.
Then $Y\notin \mathcal{B}_{[\emptyset, X_0]}\cup \mathcal{B}_{[X_3, [n]]}$ and $Y\nsim X_3$.
By Lemma~\ref{lem:TL1} with $X=X_0$, $Z=X_3$, and $W=Y$, there exists a set $Y_1 \in \mathcal{B}_{(X_0, X_3)}$ such that $Y_1 \nsim Y$.
By Claim \ref{Case1claim1}, $c(Y_1)=2$.
If $c(Y)=1$, then there is a rainbow induced copy of $\mathcal{B}_2$ by $\{X_0, Y, Y_1, [n]\}$;
if $c(Y)=2$, then there is a rainbow induced copy of $\mathcal{B}_2$ by $\{\emptyset, Y, X_3, [n]\}$;
if $c(Y)=3$, then there is a rainbow induced copy of $\mathcal{B}_2$ by $\{\emptyset, Y, Y_1, [n]\}$, a contradiction.
Therefore, $c(Y)= 4$.

We next suppose that $Y \in \mathcal{B}_{(X_3, [n]]}$. Since $c([n]) = 4$, we may assume $Y \neq [n]$.
We take $Y_2=(X_3\setminus \{a\})\cup \{b\}$, where $a\in X_3\setminus X_0$ and $b\in Y\setminus X_3$. Thus, $Y_2 <Y$, $Y_2\nsim X_3$, and $Y_2\in \mathcal{B}_{[X_0, [n]]}\setminus \left(\mathcal{B}_{[X_0, X_3]}\cup \mathcal{B}_{[X_3,[n]]}\right)$. By the previous paragraph, we have $c(Y_2)=4$.
By Lemma~\ref{lem:TL1} with $X=X_0$, $Z=X_3$, and $W=Y_2$, there exists $Y_3 \in \mathcal{B}_{(X_0, X_3)}$ with $Y_3 \nsim Y_2$. 
By Claim \ref{Case1claim1}, $c(Y_3)=2$. Note that $Y_3<Y$.
If $c(Y)=1$, then there is a rainbow induced copy of $\mathcal{B}_2$ by $\{X_0, Y_2, Y_3, Y\}$;
if $c(Y)=2$, then there is a rainbow induced copy of $\mathcal{B}_2$ by $\{\emptyset, Y_2, X_3, Y\}$;
if $c(Y)=3$, then there is a rainbow induced copy of $\mathcal{B}_2$ by $\{\emptyset, Y_2, Y_3, Y\}$, 
a contradiction. 
Therefore, $c(Y)= 4$, as desired.

In the last of this subcase, we prove that $c(Y)\in\{ 1, 4\}$
for each set $Y\in \mathcal{B}_n\setminus (\mathcal{B}_{[\emptyset, X_3]}\cup \mathcal{B}_{[X_0, [n]]})$.
Let $Y\in \mathcal{B}_n\setminus (\mathcal{B}_{[\emptyset, X_3]}\cup \mathcal{B}_{[X_0, [n]]})$. Then $Y\nsim X_3$ and $Y\notin \mathcal{B}_{[\emptyset, X_0]}\cup \mathcal{B}_{[X_3, [n]]}$.
By Lemma \ref{lem:TL1} with $X=X_0$, $Z=X_3$, and $W=Y$, there exists a set $W_0\in \mathcal{B}_{(X_0, X_3)}$ such that $Y\nsim W_0$. By Claim \ref{Case1claim1}, $c(W_0)=2$.
If $c(Y)=2$ or $c(Y)=3$, then there is a rainbow induced copy of $\mathcal{B}_2$ by $\{\emptyset, Y, X_3, [n]\}$ or $\{\emptyset, Y, W_0, [n]\}$, respectively,
a contradiction.
Therefore, $c(Y)\in\{ 1,  4\}$.

Thus, $c$ is of Type $2$, which completes the proof of Subcase 1.1.

\begin{subcase}
$c(X_0)=4$.   
\end{subcase}

In this subcase, we will show that $c$ is of Type $3$-$1$ with $X_0$ and $Y_0 = X_3$.
%Since $X_0 < X_2 < X_3$, we have $|X_3| \geq |X_0| + 2$.

We first claim that $c(Y)=4$ for every
$Y \in \mathcal{B}^{X_3\uparrow}_{(X_0, [n]]}\setminus \mathcal{B}_{(X_0, X_3]}$.
Note that $c([n])=4$. Then let $Y\in \mathcal{B}^{X_3\uparrow}_{(X_0, [n]]}\setminus (\mathcal{B}_{(X_0, X_3]}\cup\{[n]\})$.
Suppose first that $Y\notin \mathcal{B}_{(X_3, [n])}$. Then $Y\nsim X_3$.
Because
$\mathcal{B}^{X_3\uparrow}_{(X_0, [n]]}
=\{X\in \mathcal{B}_{(X_0, [n]]}: X\cap (X_3\setminus X_0)\neq \emptyset\}$, we can
choose $a\in Y\cap (X_3\setminus X_0)$ and let $Y_2=X_0\cup\{a\}$.
Then $Y_2\in \mathcal{B}_{(X_0, X_3)}$. 
Since $a\in Y$ and $X_0\subseteq Y$, we have $Y_2<Y$.
Note that $Y\notin \mathcal{B}_{[\emptyset, X_0]}\cup \mathcal{B}_{[X_3, [n]]}$.
By Lemma~\ref{lem:TL1} with $X=X_0$, $Z=X_3$, and $W=Y$, there exists $Y'_2\in \mathcal{B}_{(X_0, X_3)}$ with $Y'_2\nsim Y$.
By Claim~\ref{Case1claim1}, $c(Y_2)=c(Y'_2)=2$.
If $c(Y)=1$, then there is a rainbow induced copy of $\mathcal{B}_2$ by $\{Y_2, Y, X_3, [n]\}$;
if $c(Y)=2$, then there is a rainbow induced copy of $\mathcal{B}_2$ by $\{\emptyset, Y, X_3, [n]\}$;
if $c(Y)=3$, then there is a rainbow induced copy of $\mathcal{B}_2$ by $\{\emptyset, Y, Y'_2, [n]\}$, 
a contradiction.
Therefore, $c(Y)= 4$.

Suppose next that $Y \in \mathcal{B}_{(X_3, [n])}$.
Choose $a\in X_3\setminus X_0$, $b\in X_3\setminus (X_0\cup \{a\})$, and $c\in Y\setminus X_3$. Then let $Y_3=X_0\cup \{a\}$ and $Y_4=(X_3\setminus \{b\})\cup \{c\}$.
Thus, $Y_3\in \mathcal{B}_{(X_0, X_3)}$, $Y_4\in \mathcal{B}_{[Y_3, Y)}$ and $Y_4\nsim X_3$.
By Claim \ref{Case1claim1}, $c(Y_3)=2$.
Since $Y_4 \in \mathcal{B}_{[Y_3, Y)}\subseteq \left(\mathcal{B}^{X_3\uparrow}_{(X_0, [n]]}\setminus \mathcal{B}_{(X_0, X_3]}\right)\setminus \mathcal{B}_{[X_3, [n]]}$, it follows from the previous paragraph that $c(Y_4)=4$.
If $c(Y)=1$, then there is a rainbow induced copy of $\mathcal{B}_2$ by $\{Y_3, Y_4, X_3, Y\}$;
if $c(Y)=2$, then there is a rainbow induced copy of $\mathcal{B}_2$ by $\{\emptyset, Y_4, X_3, Y\}$, a contradiction.
Recall that $Y_4\nsim X_3$, and hence $Y_4\notin \mathcal{B}_{[\emptyset, X_0]}\cup \mathcal{B}_{[X_3, [n]]}$. By Lemma \ref{lem:TL1} with $X=X_0$, $Z=X_3$, and $W=Y_4$, there exists a set
$Y_5 \in \mathcal{B}_{(X_0, X_3)}$
such that $Y_5\nsim Y_4$, and $c(Y_5)=2$, by Claim \ref{Case1claim1}.
If $c(Y)=3$, then there is a rainbow induced copy of $\mathcal{B}_2$ by $\{\emptyset, Y_4, Y_5, Y\}$, a contradiction.
Therefore, $c(Y) = 4$, 
as claimed.

In the last of subcase 1.2, 
we prove that $c(Y) \in \{1, 4\}$ for each set $Y \in \mathcal{B}_n\setminus (\mathcal{B}_{[\emptyset, X_3]}\cup \mathcal{B}^{X_3\uparrow}_{(X_0, [n]]})$.
Let $Y \in \mathcal{B}_n\setminus (\mathcal{B}_{[\emptyset, X_3]}\cup \mathcal{B}^{X_3\uparrow}_{(X_0, [n]]})$.
Then $Y\nsim X_3$ and $Y\notin \mathcal{B}_{[\emptyset, X_0]}\cup \mathcal{B}_{[X_3, [n]]}$. By Lemma \ref{lem:TL1} with $X=X_0$, $Z=X_3$, and $W=Y$, there exists a set $Y'\in \mathcal{B}_{(X_0, X_3)}$ such that $Y'\nsim Y$.
By Claim \ref{Case1claim1}, $c(Y') = 2$.
If $c(Y)=2$, then there is a rainbow induced copy of $\mathcal{B}_2$ by $\{\emptyset, Y, X_3, [n]\}$;
if $c(Y)=3$, then there is a rainbow induced copy of $\mathcal{B}_2$ by $\{\emptyset, Y, Y', [n]\}$, 
a contradiction. Therefore, $c(Y)\in\{ 1, 4\}$.

Then $c$ is of Type $3$-$1$, 
as desired.

\begin{case}
There exists
$Y_0 \in \mathcal{B}_{(X_3, [n])}$ with $c(Y_0)\in\{1, 2\}$.
\end{case}

This case is symmetric to Case~\ref{case1}. Note that Type 3-2 is symmetric to Type 3-1.  
Therefore, we conclude that $c$ is of Type 2 or Type 3-2.

\begin{case}
All sets in $\mathcal{B}_{[\emptyset, X_2)}$ are colored $1$ and all sets in $\mathcal{B}_{(X_3, [n]]}$ are colored $4$. 
\end{case}

Let $\mathcal{X}_2=\{X\in \mathcal{B}_n : c(X)=2\}$ and $\mathcal{X}_3=\{X\in \Bn : c(X)=3\}$. It is clear that $X_2\in \mathcal{X}_2$ and $X_3\in \mathcal{X}_3$.
Recall that $c(X) \neq 2$ for each $X$ with $|X|<|X_2|$
and $c(X) \neq 3$ for each $X$ with $|X|>|X_3|$.
Therefore, $|X|\geq |X_2|$ for each $X\in \mathcal{X}_2$ and $|X|\leq |X_3|$ for each $X\in \mathcal{X}_3$.  

We will later divide this case further into three subcases depending on $|\mathcal{X}_2|$ and $|\mathcal{X}_3|$, 
and show that $c$ is of Type $4$-$1$ or $4$-$2$, 
respectively.
Before that we give the following claim.
\begin{claim}\label{claim12}
$|\mathcal{X}_2|=1$ or $|\mathcal{X}_3|=1$.
\end{claim}
\begin{proof}
Suppose that $|\mathcal{X}_2| \geq 2$ and $|\mathcal{X}_3| \geq 2$. Then there exist distinct sets $X_2, X_2' \in \mathcal{X}_2$ and $X_3, X_3' \in \mathcal{X}_3$. Recall that $|X| \geq |X_2|$ for each $X \in \mathcal{X}_2$ and $|X| \leq |X_3|$ for each $X \in \mathcal{X}_3$. Therefore, $|X_2'| \geq |X_2|$ and $|X_3'| \leq |X_3|$. By Lemma~\ref{lem-x-y}, we have $X_2 \sim X_3'$, $X_2' \sim X_3'$, and $X_2' \sim X_3$. Since $X_2 < X_3$ and all sets in $\mathcal{B}_{[\emptyset, X_2)}$ are colored $1$ while all sets in $\mathcal{B}_{(X_3, [n]]}$ are colored $4$, we obtain that $X_3' \not\leq X_2$ and $X_3 \not\leq X_2'$.

Assume that $|X_3| = |X_2| + 1$. Since $X_3' \not\leq X_2$ and $X_3' \sim X_2$, it must be that $X_2 < X_3'$. Since $|X_3'| \leq |X_3|$, we have $|X_3'| = |X_2| + 1 = |X_3|$. Similarly, since $X_3 \not\leq X_2'$ and $X_2' \sim X_3$, we deduce that $X_2' < X_3$, implying $|X_2'| = |X_2| = |X_3'|-1$. The relation $X_2' \sim X_3'$ together with the fact that $|X_2'| < |X_3'|$ implies that $X_2' < X_3'$. Consequently, 
$X_2<X_3$, $X_2<X_3'$, $X'_2<X_3$, and $X_2'<X_3'$, and hence
both $X_2 \cup X_2' \leq X_3$ and $X_2 \cup X_2' \leq X_3'$ hold. Since $|X_2'| = |X_2|$ and $|X_3| = |X_2| + 1$, it follows that $X_3 = X_2 \cup X_2'$. Thus, $X_3 \leq X_3'$, contradicting the fact that $|X_3| = |X_3'|$ and $X_3\neq X_3'$. 

Therefore, $|X_3| > |X_2| + 1$. 
Recall that $|X_2'| \geq |X_2|$ and $X_3 \not\leq X_2'$, and hence $X_2'\notin \mathcal{B}_{[\emptyset, X_2]}\cup \mathcal{B}_{[X_3, [n]]}$.
By Lemma \ref{lem:TL1} with $X=X_2$, $Z=X_3$, and $W=X_2'$, we can find a set $Y_1 \in \mathcal{B}_{(X_2, X_3)}$ with $|Y_1| = |X_2| + 1$ such that $Y_1 \nsim X_2'$. Similarly, we can also find a set $Y_2 \in \mathcal{B}_{(X_2, X_3)}$ with $|Y_2| = |X_2| + 1$ such that $Y_2 \nsim X_3'$. If $c(Y_1) = 3$, then there is a rainbow induced copy of $\mathcal{B}_2$ by $\{\emptyset, Y_1, X_2', [n]\}$; if $c(Y_1) = 4$, then there is a rainbow induced copy of $\mathcal{B}_2$ by $\{\emptyset, Y_1, X_2', X_3\}$, a contradiction. Therefore, $c(Y_1) \in \{1, 2\}$. If $c(Y_2) = 1$, then there is a rainbow induced copy of $\mathcal{B}_2$ by $\{X_2, Y_2, X_3', [n]\}$; if $c(Y_2) = 2$, then there is a rainbow induced copy of $\mathcal{B}_2$ by $\{\emptyset, Y_2, X_3', [n]\}$, a contradiction. Hence, $c(Y_2) \in \{3, 4\}$. 
Since $c(Y_1)\in\{1, 2\}$ and $c(Y_2)\in\{3, 4\}$, we have $Y_1\neq Y_2$.
Moreover $|Y_1|=|Y_2|=|X_2|+1$ implies $Y_1\nsim Y_2$. As $Y_1, Y_2\in\mathcal{B}_{(X_2, X_3)}$, we also have
$\emptyset<X_2<Y_1, Y_2<X_3<[n]$.
Choose the bottom set
$B\in\{\emptyset, X_2\}$ so that 
$c(B) \neq c(Y_1)$, 
and the top set
$T\in\{X_3, [n]\}$ so that $c(T)\neq c(Y_2)$.
Then there is a rainbow induced copy of $\mathcal{B}_2$ by $\{B, Y_1, Y_2, T\}$, a contradiction.
Therefore, the claim holds.
\end{proof}
Based on Claim \ref{claim12}, 
we consider the following three cases.
\begin{subcase}\label{subcase31}
$|\mathcal{X}_2|=1$ and $|\mathcal{X}_3|> 1$.    
\end{subcase}
In this subcase, we show that $c$ is of Type~4-1
with $\{X_0\} = \mathcal{X}_2 = \{X_2\}$ and $\mathcal{Y}_0 = \mathcal{X}_3$.

For each $X\in \mathcal{X}_3$, since $c(X_2)=2$ and $c(X)=3$, it follows from Lemma~\ref{lem-x-y} that $X\sim X_2$. Since all sets in $\mathcal{B}_{[\emptyset, X_2)}$ are colored $1$, $X> X_2$ for each $X\in \mathcal{X}_3$.

\begin{claim}\label{claim-subcase-3-1-1}
$c(Y) = 1$ for each $Y \in \bigcup_{X \in \mathcal{X}_3} \left( \mathcal{B}_{[\emptyset, X)} \setminus \mathcal{B}_{[X_2, X)} \right)$.   
\end{claim}

\begin{proof}
Let $Y \in \bigcup_{X \in \mathcal{X}_3} \left( \mathcal{B}_{[\emptyset, X)} \setminus \mathcal{B}_{[X_2, X)} \right)$.
Since $c(Y) = 1$ if $Y \in \mathcal{B}_{[\emptyset, X_2)}$, it suffices to consider $Y \notin \mathcal{B}_{[\emptyset, X_2)}.$
Note that  $Y\nsim X_2$.
Observe that there exists a set $X_3' \in \mathcal{X}_3$ such that $\emptyset < Y < X_3' < [n]$. 
If $c(Y) = 3$, then $\{\emptyset, Y, X_2, [n]\}$ forms a rainbow induced copy of $\mathcal{B}_2$; 
if $c(Y) = 4$, then $\{\emptyset, Y, X_2, X_3'\}$ forms a rainbow induced copy of $\mathcal{B}_2$, a contradiction.
Hence, $c(Y) \in \{1, 2\}$. 
Since $\mathcal{X}_2 = \{X_2\}$, we further conclude that $c(Y) = 1$ if 
$Y \notin \mathcal{B}_{[\emptyset, X_2)}$.
\end{proof}

\begin{claim}\label{claim-subcase-3-1-2}
$c(Y)=4$ for any $Y\in \mathcal{B}_{(X_2, [n]]}\setminus \mathcal{X}_3$.
\end{claim}
\begin{proof}
Let $Y\in \mathcal{B}_{(X_2, [n]]}\setminus \mathcal{X}_3$.
Note that $c(Y) = 4$ if $Y \in \mathcal{B}_{(X_3, [n]]}$. Therefore, we only need to consider $Y\notin \mathcal{B}_{(X_3, [n]]}$. 
We first suppose that $Y \in \mathcal{B}_{(X_2, X_3)} \setminus \mathcal{X}_3$. If $\mathcal{B}_{(X_2, X_3)}\subseteq \mathcal{X}_3$, then there are no sets in $\mathcal{B}_{(X_2, X_3)} \setminus \mathcal{X}_3$. 

Now suppose that $\mathcal{B}_{(X_2, X_3)}\setminus \mathcal{X}_3\neq \emptyset$. In particular, $|X_3|>|X_2|+1$. By the definitions of $\mathcal{X}_2$ and $\mathcal{X}_3$, we know that $c(Y)\in\{1, 4\}$.
Let $\mathcal{L} = \{X \in \mathcal{B}_{(X_2, X_3)} : |X| = |X_2| + 1\}$. Assume $\mathcal{L} \subseteq \mathcal{X}_3$. By Lemma \ref{lem:TL1} with $X=X_2$, $Z=X_3$, and $W=Y$, 
there exists $X_3' \in \mathcal{L}$ such that $Y \nsim X_3'$. If $c(Y) = 1$, then the set $\{X_2, Y, X_3', [n]\}$ forms a rainbow induced copy of $\mathcal{B}_2$, a contradiction. Therefore, $c(Y) = 4$ in this case.

Assume instead $\mathcal{L} \not\subseteq \mathcal{X}_3$.
We will show that there exists a set $Y_1 \in \mathcal{L} \setminus \mathcal{X}_3$ such that $c(Y_1)=4$.
If $\mathcal{L} \cap \mathcal{X}_3 \neq \emptyset$, then there exist sets $Y_1 \in \mathcal{L} \setminus \mathcal{X}_3$ and $X_3' \in \mathcal{L} \cap \mathcal{X}_3$ such that $Y_1 \nsim X_3'$. 
Otherwise $\mathcal{L} \cap \mathcal{X}_3 = \emptyset$. 
Recall that $c(X)\neq 3$ for each $X$ with $|X|>|X_3|$. Therefore, $|X|\leq |X_3|$ and $X>X_2$ for any $X\in \mathcal{X}_3$. Thus, there exists $X_3' \in \mathcal{X}_3 \setminus \{X_3\}$ such that 
$X'_3\notin \mathcal{B}_{[\emptyset, X_2]}\cup\mathcal{B}_{[X_3, [n]]}$. By Lemma \ref{lem:TL1} with $X=X_2$, $Z=X_3$, and $W=X'_3$, 
there exists $Y_1 \in \mathcal{L}$ such that $Y_1 \nsim X_3'$. 
If $c(Y_1) = 1$, then $\{X_2, Y_1, X_3', [n]\}$ forms a rainbow induced copy of $\mathcal{B}_2$, a contradiction, and so $c(Y_1) = 4$.

Suppose $Y \in \mathcal{L} \setminus \left(\mathcal{X}_3\cup\{Y_1\}\right)$, we observe that $Y\nsim Y_1$. If $c(Y) = 1$, then $\{X_2, Y_1, Y, X_3\}$ would form a rainbow induced copy of $\mathcal{B}_2$, a contradiction. Hence, $c(Y) = 4$. 
Suppose $Y \in \mathcal{B}_{(X_2, X_3)} \setminus \left(\mathcal{L}\cup \mathcal{X}_3\right)$. By Lemma \ref{lem:TL1} with $X=X_2$, $Z=X_3$, and $W=Y$, there exists $Y_1' \in \mathcal{L}$ such that $Y \nsim Y_1'$. If $c(Y) = 1$, then there is a rainbow induced copy of $\mathcal{B}_2$ by $\{X_2, Y_1', Y, X_3\}$ (when $Y_1' \in \mathcal{L} \setminus \mathcal{X}_3$) or $\{X_2, Y_1', Y, [n]\}$ (when $Y_1' \in \mathcal{L} \cap \mathcal{X}_3$), a contradiction. Therefore, $c(Y) = 4$.

Next, suppose $Y\notin \mathcal{B}_{(X_2, X_3)}\setminus \mathcal{X}_3$. Note that $Y\notin \mathcal{B}_{(X_3, [n]]}$. We have $\emptyset< X_2<Y\nsim X_3<[n]$.
From the definitions of $\mathcal{X}_2$ and $\mathcal{X}_3$, we know that $c(Y)\in \{1, 4\}$.
If $c(Y)=1$, then there is a rainbow induced copy of $\mathcal{B}_2$ by $\{X_2, Y, X_3, [n]\}$, a contradiction. Therefore, $c(Y)=4$, and hence $c(Y)=4$ for each $Y\in \mathcal{B}_{(X_2, [n]]}\setminus \mathcal{X}_3$, as desired.
\end{proof}

By Claims~\ref{claim-subcase-3-1-1} and \ref{claim-subcase-3-1-2}, together with the definitions of $X_2$ and $\mathcal{X}_3$, the colors of the sets in $\mathcal{B}_{[X_2,[n]]}\cup\bigl(\bigcup_{X\in\mathcal{X}_3}\mathcal{B}_{(\emptyset,X)}\bigr)$ are determined.
Finally, suppose $Y\in \Bn\setminus\left(\mathcal{B}_{[X_2, [n]]}\cup (\bigcup_{X\in\mathcal{X}_3} \mathcal{B}_{(\emptyset, X)})\right)$. 
Note that $Y\nsim X_2$ and $Y\nsim X_3$. 
Recall that $X> X_2$ for each $X\in \mathcal{X}_3$ and $|\mathcal{X}_2|=1$. Therefore, $Y\notin \mathcal{X}_2\cup\mathcal{X}_3$, and hence $c(Y)\in \{1, 4\}$.
Then $c$ is of Type 4-1 under the condition $|\mathcal{Y}_0|\neq 1$, as desired.

\begin{subcase}
$|\mathcal{X}_2|> 1$ and $|\mathcal{X}_3|=1$.
\end{subcase}
This subcase is symmetric to Subcase~\ref{subcase31}. Note that Type 4-2 is symmetric to Type 4-1.  
Therefore, we conclude that $c$ is of Type 4-2 under the condition $|\mathcal{X}_0|\neq 1$.

\begin{subcase}
$|\mathcal{X}_2|=1$ and $|\mathcal{X}_3|=1$.    
\end{subcase}
In this subcase, we show that $c$ is of Type~4-1 under the condition $|\mathcal{Y}_0| = 1$, where $X_0 = X_2$ and $\mathcal{Y}_0 = \mathcal{X}_3 = \{X_3\}$, or of Type~4-2 under the condition $|\mathcal{X}_0| = 1$, where $Y_0 = X_3$ and $\mathcal{X}_0 = \mathcal{X}_2 = \{X_2\}$. 
Specifically, Types~$4$-$1$ and $4$-$2$ differ only in the color on $\mathcal{B}_{(X_2,X_3)}$. 

Since $\mathcal{X}_2=\{X_2\}$ and $\mathcal{X}_3=\{X_3\}$, it follows that $c(Y)\in \{1, 4\}$ for each $Y\in \mathcal{B}_n\setminus\{X_2, X_3\}$.
Therefore, we just need to consider the colors of the sets in $(\mathcal{B}_{[\emptyset, X_3]}\cup \mathcal{B}_{[X_2, [n]]})\setminus \{X_2,X_3\}$.

Let $Y\in \mathcal{B}_{[\emptyset, X_3)}\setminus \mathcal{B}_{[X_2, X_3)}$.
Note that $c(Y)=1$ if $Y\in\mathcal{B}_{[\emptyset, X_2)}$.
Suppose $Y\notin \mathcal{B}_{[\emptyset, X_2)}$.
Then $Y\nsim X_2$.
If $c(Y)=4$, then there is a rainbow induced copy of $\mathcal{B}_2$ by $\{\emptyset, X_2, Y, X_3\}$, a contradiction. Thus, $c(Y)=1$.
Next, let $Y\in \mathcal{B}_{[X_2, [n]]}\setminus \mathcal{B}_{[X_2, X_3]}$.
Note that $c(Y)=4$ if $Y\in\mathcal{B}_{(X_3, [n]]}$.
Suppose $Y\notin \mathcal{B}_{(X_3, [n]]}$.
Then $Y\nsim X_3$.
If $c(Y)=1$, then there is a rainbow induced copy of $\mathcal{B}_2$ by $\{X_2, Y, X_3, [n]\}$, a contradiction. Thus, $c(Y)=4$.

In the end, we claim that all sets in $\mathcal{B}_{(X_2,X_3)}$ have the same color, which is either $1$ or $4$.
If $|X_3|=|X_2|+1$, then $\mathcal{B}_{(X_2, X_3)}=\emptyset$, and the claim trivially holds.
Suppose that $|X_3|>|X_2|+1$. 
Let $\mathcal{D}_i=\{X\in\mathcal{B}_{(X_2, X_3)} :  |X|=i\}$, where $|X_2|< i<|X_3|$.
Considering any two sets $Y, Y'\in\mathcal{D}_i$, we have $\emptyset<X_2<Y\nsim Y'<X_3<[n]$. If $c(Y)=1$ and 
$c(Y')=4$, then 
there is a rainbow induced copy of $\mathcal{B}_2$ by $\{X_2, Y, Y', X_3\}$, a contradiction. Therefore, for each $i$ with $|X_2|<i<|X_3|$, all sets in $\mathcal{D}_i$ have the same color, which is either $1$ or $4$. 

For any two integers $i, j$ with $|X_2|<i\neq j<|X_3|$, if $c(Y)=1$ for each $Y\in \mathcal{D}_i$ and $c(Y')=4$ for each $Y'\in \mathcal{D}_j$, then we can find two sets $Y_i\in \mathcal{D}_i$ and $Y_j\in \mathcal{D}_j$ such that $Y_i\nsim Y_j$, and hence there is a rainbow induced copy of $\mathcal{B}_2$ by $\{X_2, Y_i, Y_j, X_3\}$, a contradiction. Therefore, all sets in $\mathcal{B}_{(X_2,X_3)}$ have the same color, which is either $1$ or $4$.

Thus, $c$ is either of Type~4-1 under the condition $|\mathcal{Y}_0| = 1$, or of Type~4-2 under the condition $|\mathcal{X}_0| = 1$.
\end{proof}

\subsection{An exact $k$-coloring $c$ without a rainbow induced copy of $\mathcal{B}_2$ with $c(\emptyset)= c([n])$ and $k\geq 4$}\label{subsection25}

Before stating the complete characterization in Theorem \ref{Structural-B2-1} in Section 2.6, we 
introduce one more exact $k$-coloring of $\mathcal{B}_n$ of Type~5, where $c(\emptyset)=c([n])$.
Let $W, Y$ be two sets with $\emptyset< W<Y<[n]$. Recall that
$$
\mathcal{B}^{Y\uparrow}_{(W, [n]]}=\bigcup_{Z \in \mathcal{B}_{(W, Y)}} \mathcal{B}_{[Z, [n]]}=\{Z\in \mathcal{B}_{(W, [n]]} :  Z\cap (Y\setminus W)\neq \emptyset\}, 
$$
and 
$$\mathcal{B}^{W\downarrow}_{[\emptyset, Y)}=\bigcup_{Z \in \mathcal{B}_{(W, Y)}} \mathcal{B}_{[\emptyset, Z]}=\{Z\in \mathcal{B}_{[\emptyset, Y)} :  Y\setminus W\not\subseteq Z\}.$$\\[2mm]
\emph{Type~5}: The maximum length of any rainbow chain in $\mathcal{B}_n$ is $4$, and the following two conditions holds.
\begin{itemize}
\item[$(1)$] For any rainbow $4$-chain $(W, X, Y, [n])$, 
all sets in the following families are monochromatically colored
so that no two sets from different families share the same color:
$$
\left(\mathcal{B}^{Y\uparrow}_{(W, [n]]}\cup \mathcal{B}^{W\downarrow}_{[\emptyset, Y)}\right)\setminus \mathcal{B}_{[W, Y]}, \quad
\{W\}, \quad
\mathcal{B}_{(W, Y)}, \quad
\{Y\}.
$$
Moreover, each set in $\mathcal{B}_{[\emptyset, Y)}\setminus \mathcal{B}^{W\downarrow}_{[\emptyset, Y)}$
has the same color as the sets in the first or the last family, and each set in $\mathcal{B}_{(W, [n]]}\setminus \mathcal{B}^{Y\uparrow}_{(W, [n]]}$
has the same color as the sets in the first or the second family.

\item[$(2)$] For any rainbow $3$-chain $(W, Y, [n])$ that cannot be extended to a rainbow $4$-chain by adding another set, the union 
$\mathcal{B}_{[\emptyset, Y]} \cup \mathcal{B}_{[W, [n]]}$
contains $3$ colors.
\end{itemize}

Let $W, Y$ be two sets such that $(W, Y, [n])$ is a rainbow chain. 
Note that an exact $k$-coloring of Type 5
is symmetric by exchanging the roles of $W$ and $Y$, 
those of $\emptyset$ and $[n]$, and so on, and there are at most $4$ colors appearing in $\mathcal{B}_{[\emptyset, Y]} \cup \mathcal{B}_{[W, [n]]}$.
By the definitions of 
$\mathcal{B}^{Y\uparrow}_{(W, [n]]}$
and $\mathcal{B}^{W\downarrow}_{[\emptyset, Y)}$, we have $\mathcal{B}_{(W, [n]]}\setminus \mathcal{B}^{Y\uparrow}_{(W, [n]]}=\{Z\in \mathcal{B}_{(W, [n]]} :  Z\cap (Y\setminus W)= \emptyset\}$ and $\mathcal{B}_{[\emptyset, Y)}\setminus \mathcal{B}^{W\downarrow}_{[\emptyset, Y)}=\{Z\in \mathcal{B}_{[\emptyset, Y)} :  Y\setminus W\subseteq Z\}$.

\begin{remark}
In Type~5(1), we have the following.
\begin{itemize}
    \item[$(i)$] The same conclusion as in Type 5 $(1)$ also holds for any rainbow $4$-chain of the form $(\emptyset, W, X, Y)$.
    \item[$(ii)$] The colors of sets in $\mathcal{B}_{[\emptyset, Y]}$ are identical to those in Type~4--1, 
after replacing $[n]$ by $Y$, $X_0$ by $W$, and $\mathcal{Y}_0$ by $\mathcal{B}_{(W, Y)}$.
    \item[$(iii)$] The colors of sets in $\mathcal{B}_{[W,[n]]}$ are identical to those in Type~4--2, after replacing $\emptyset$ by $W$, $Y_0$ by $Y$, and $\mathcal{X}_0$ by $\mathcal{B}_{(W, Y)}$.
\end{itemize}
\end{remark}

\begin{lemma}\label{emptyset=n}
Let $k, n$ be two integers with $k\geq 4$ and $n>2$.
Consider an exact $k$-coloring $c$ of $\mathcal{B}_n$
with $c(\emptyset)= c([n])$.
Then there is no rainbow induced copy of $\mathcal{B}_2$
if and only if the coloring $c$ is of Type $5$.
\end{lemma}

Before proving Lemma \ref{emptyset=n}, we show the following lemma, which is useful for our proofs.
\begin{lemma}\label{clm2}
Let $k, n$ be two integers with $n>2$ and $k\geq 4$. 
Consider an exact $k$-coloring $c$ of $\mathcal{B}_n$
with $c(\emptyset)= c([n])$. Then the maximum length of any rainbow chain in $\mathcal{B}_n$ is at most $4$.
\end{lemma}
\begin{proof}
Suppose, for contradiction, that there exists a rainbow $5$-chain
$(Y_0, Y_1, Y_2, Y_3, Y_4)$ in $\mathcal{B}_n$.
If $c([n])\notin \{c(Y_0), c(Y_1), c(Y_2), c(Y_3), c(Y_4)\}$, replace $Y_4$ by $[n]$; otherwise pick $j$ with $c(Y_j)=c([n])$ and delete $Y_j$, then append $[n]$ at the top.
In either case we obtain a rainbow $5$-chain $(W_0, W_1, W_2, W_3, [n])$
in $\mathcal{B}_n$.

Therefore, the colors of sets in $\mathcal{B}_{[W_0,[n]]}$ are identical to those in Type~1, after replacing $\emptyset$ by $W_0$, $X_0$ by $W_1$, and $Y_0$ by $W_3$, that is, 
all sets in the following families are monochromatically colored so that no two sets from different families share the same color:
$$
\mathcal{B}_{[W_0, W_3]} \setminus \mathcal{B}_{[W_1, W_3]}, \quad \{W_1\}, \quad \mathcal{B}_{(W_1, W_3)}, \quad \{W_3\}, \quad \mathcal{B}_{[W_1, [n]]} \setminus \mathcal{B}_{[W_1, W_3]}.$$
Moreover, each set in $\mathcal{B}_{[W_0, [n]]} \setminus \left(\mathcal{B}_{[W_0, W_3]}\cup \mathcal{B}_{[W_1, [n]]}\right)$ has the same color as the sets in the first or the last
family.

Let $a\in W_3\setminus W_1$ and $b\in W_1$, and set $W'_1=(W_1\cup\{a\})\setminus\{b\}$. Then, $W_1'\in \mathcal{B}_{[W_0, W_3]} \setminus \mathcal{B}_{[W_1, W_3]}$ and $W'_1\nsim W_1$. Thus, $c(W'_1)=c(W_0)$. Since $c(\emptyset)=c([n])$, the family $\{\emptyset, W_1', W_1, W_3\}$ forms a rainbow induced copy of $\mathcal{B}_2$, a contradiction. Hence, the maximum length of any rainbow chain in $\mathcal{B}_n$ is at most $4$.
\end{proof}

\begin{proof}[Proof of Lemma \ref{emptyset=n}]
We first prove the ``if'' part. Suppose that a coloring $c$ is of Type~5.  
Assume, for contradiction, that there exists a rainbow induced copy of $\mathcal{B}_2$ formed by sets $W_1, W_2, W_3, W_4 \in \mathcal{B}_n$, where $W_1 < W_2, W_3 < W_4$, and $W_2 \nsim W_3$.

Without loss of generality, suppose that $c(\emptyset) = 1$. Then $c([n]) = 1$. 
If $1 \notin \{c(W_1), c(W_4)\}$, then since $c(W_2)\neq c(W_3)$, either $(W_1, W_2, W_4, [n])$ or $(W_1, W_3, W_4, [n])$ forms a rainbow $4$-chain.  
Since $c$ is of Type~5, all sets in $\mathcal{B}_{(W_1, W_4)}$ are monochromatically colored, which implies $c(W_2) = c(W_3)$, a contradiction.
Therefore, we must have $1 \in \{c(W_1), c(W_4)\}$.  
We now consider the following two cases.

\begin{ccase}\label{subcase2-1}
$c(W_1)=1$.
\end{ccase}
Since $c([n])=1$, $(W_2, W_4, [n])$ forms a rainbow $3$-chain.
If $(W_2, W_4, [n])$ cannot be extended to a rainbow $4$-chain by adding another set, then by Type~5 (2), the union $\mathcal{B}_{[\emptyset, W_4]} \cup \mathcal{B}_{[W_2, [n]]}$
must contain exactly three colors: $1$, $c(W_2)$, and $c(W_4)$. This implies
$c(W_3) \in \{1, c(W_2), c(W_4)\},$
a contradiction.

Thus, there exists a set $X_0$ such that $\{X_0, W_2, W_4, [n]\}$ forms a rainbow $4$-chain, which must be one of 
$(X_0, W_2, W_4, [n]),  (W_2, X_0, W_4, [n])$, or $(W_2, W_4, X_0, [n]).$
Without loss of generality, suppose $c(W_2) = 2$, $c(W_4) = 4$, and $c(X_0) = 3$. Since $W_1, W_2, W_3 \in \mathcal{B}_{[\emptyset, W_4)}$, it suffices to analyze the colors of sets in $\mathcal{B}_{[\emptyset, W_4]}$.

If the rainbow chain is $(X_0, W_2, W_4, [n])$, then, by Type~5 (1) with
$W = X_0$, $X = W_2$, and $Y = W_4$, we see that among the sets in
$\mathcal{B}_{[\emptyset, W_4]}$, only the sets in $\mathcal{B}_{(X_0, W_4)}$
are colored $2$, and only the set $X_0$ is colored $3$.
Since $c(W_1)=1$ and $\{W_1, W_2, W_3, W_4\}$ forms a rainbow induced copy of $\mathcal{B}_2$, we have $c(W_3)=3$, and hence $W_3=X_0<W_2$, a contradiction.

If the rainbow chain is $(W_2, X_0, W_4, [n])$, then by Type~5 (1) with $W=W_2$, $X=X_0$, and $Y=W_4$, we see that among the sets in
$\mathcal{B}_{[\emptyset, W_4]}$, only the sets in $\mathcal{B}_{(W_2, W_4)}$ are colored $3$. Since $c(W_1)=1$ and $\{W_1, W_2, W_3, W_4\}$ forms a rainbow induced copy of $\mathcal{B}_2$, we have $c(W_3)=3$, and hence $W_3\in \mathcal{B}_{(W_2, W_4)}$, a contradiction.

If the rainbow chain is $(W_2, W_4, X_0, [n])$, then by Type~5 (1) with $W=W_2$, $X=W_4$, and $Y=X_0$, all sets in
$
\left(\mathcal{B}^{X_0\uparrow}_{(W_2, [n]]}\cup \mathcal{B}^{W_2\downarrow}_{[\emptyset, X_0)}\right)\setminus \mathcal{B}_{[W_2, X_0]}
$ 
are colored $1$. Since $W_4\in \mathcal{B}_{(W_2, X_0)}$, it follows that all sets in $\mathcal{B}_{[\emptyset, W_4]}\setminus \mathcal{B}_{[W_2, W_4]}$ are colored $1$.  Note that $W_3\in \mathcal{B}_{[\emptyset, W_4]}\setminus \mathcal{B}_{[W_2, W_4]}$. Thus, $c(W_3)=1$, a contradiction.

\begin{ccase}
$c(W_4)=1$.
\end{ccase}
If $(W_1, W_2, [n])$ cannot be extended to a rainbow $4$-chain by adding another set, then by Type~5 (2), the union $\mathcal{B}_{[\emptyset, W_2]} \cup \mathcal{B}_{[W_1, [n]]}$ contains exactly three colors: $c(W_1)$, $c(W_2)$, and $c([n])$. This implies $c(W_3) \in \{c(W_1), c(W_2), c([n])\}$, a contradiction. Therefore, there exists a set $X_0$ such that $\{W_1, W_2, [n], X_0\}$ forms a rainbow $4$-chain, which must be one of $(X_0, W_1, W_2, [n])$, $(W_1, X_0, W_2, [n])$, or $(W_1, W_2, X_0, [n])$. Since $W_2, W_3, W_4 \in \mathcal{B}_{(W_1,[n]]}$, it suffices to analyze the coloring within $\mathcal{B}_{[W_1,[n]]}$. This situation is completely analogous to Case~\ref{subcase2-1}: if we replace $\emptyset$ by $W_1$ and $W_4$ by $[n]$, then the interval $\mathcal{B}_{[W_1,[n]]}$ plays exactly the same role as $\mathcal{B}_{[\emptyset,W_4]}$ in Case~\ref{subcase2-1}. Hence the above argument can be applied in exactly the same way within $\mathcal{B}_{[W_1,[n]]}$ and again leads to a contradiction.

This completes the proof of  the ``if'' part.
\medskip

We next prove the ``only if'' part. Recall that $c(\emptyset)=c([n])$ and $k\geq 4$.
By Lemma~\ref{clm2}, any rainbow chain in $\mathcal{B}_n$ has length at most $4$. 

Assume that there exists a rainbow $4$-chain $(Y_0, Y_1, Y_2, Y_3)$ in $\mathcal{B}_n$.
If $c([n])\notin\{c(Y_0), c(Y_1), c(Y_2), c(Y_3)\}$, then
$(Y_0, Y_1, Y_2, Y_3, [n])$
is a rainbow $5$-chain, contradicting Lemma~\ref{clm2}.
Hence $c([n])=c(Y_j)$ for some $j\in\{0, 1, 2, 3\}$, and replacing $Y_j$ by $[n]$ yields a rainbow $4$-chain
with maximum set $[n]$.

Fix a rainbow $4$-chain $(W, X, Y, [n])$.
Then $W<X<Y<[n]$ lies entirely in the interval $\mathcal{B}_{[W, [n]]}$, which therefore already exhibits four distinct colors (witnessed by $W, X, Y, [n]$).
Since $c(W)\neq c([n])$, it follows from Lemma \ref{k4} that the coloring of $c$ in $\mathcal{B}_{[W, [n]]}$ must be one of Types~2, 3-1, 3-2, 4-1, 4-2.
(An entirely symmetric statement holds in $\mathcal{B}_{[\emptyset, Y]}$ if one anchors the chain at $\emptyset$ instead of $[n]$.)

First, we consider the colors of the sets in $\mathcal{B}_{[W, [n]]}$.
Assume that Type 2 holds in $\mathcal{B}_{[W, [n]]}$. 
\begin{itemize}
        \item Type 2: There exist two sets $X_0, Y_0 \in \mathcal{B}_{[W, [n]]}$ with $W < X_0 < Y_0 < [n]$ and $|Y_0| \geq |X_0| + 2$
such that all sets in the following families are monochromatically colored
so that no two sets from different families share the same color:
$$ \mathcal{B}_{[W, Y_0]} \setminus \mathcal{B}_{[X_0, Y_0]}, \quad
\{X_0, Y_0\}, \quad
 \mathcal{B}_{(X_0, Y_0)}, \quad \mathcal{B}_{[X_0, [n]]} \setminus \mathcal{B}_{[X_0, Y_0]}.$$
Moreover, 
each set in 
$\mathcal{B}_{[W, [n]]} \setminus \left(\mathcal{B}_{[W, Y_0]} \cup \mathcal{B}_{[X_0, [n]]}\right)$
has the same color as
the sets in the first or the last family.
\end{itemize}
Note that each set in $\mathcal{B}_{[W, [n]]} \setminus \mathcal{B}_{[X_0, Y_0]}$ is colored either $c(W)$ or $c([n])$.
Since $(W, X, Y, [n])$ is a rainbow $4$-chain in 
$\mathcal{B}_{[W, [n]]}$, it follows from Type 2 that $X, Y\in\mathcal{B}_{[X_0, Y_0]}$. Therefore, 
either $X=X_0$ with $Y\in \mathcal{B}_{(X_0, Y_0)}$, or $Y=Y_0$ with $X\in \mathcal{B}_{(X_0, Y_0)}$.

Suppose that $X=X_0$ and $Y\in \mathcal{B}_{(X_0, Y_0)}$.
Then all sets in the families $\mathcal{B}_{[W, Y]} \setminus \mathcal{B}_{[X, Y]}$, $\{X\}$, $\mathcal{B}_{(X, Y)}$ and $\{[n]\}$ are monochromatically colored so that no two sets from different families share the same color.
We can find two elements $x \in X \setminus W$ and $y \in Y \setminus X$. Observe that $(X\cup\{y\}) \setminus \{x\} \in \mathcal{B}_{[W, Y]} \setminus \mathcal{B}_{[X, Y]}\subseteq \mathcal{B}_{[W, Y_0]} \setminus \mathcal{B}_{[X, Y_0]}$. Since $c(\emptyset) = c([n])$, the sets $(X\cup\{y\}) \setminus \{x\}$, $X$, $Y$ and $\emptyset$ all have distinct colors. 
Hence, the family $\{\emptyset, (X\cup\{y\}) \setminus \{x\}, X, Y\}$ forms a rainbow induced copy of $\mathcal{B}_2$, a contradiction.

Suppose that $Y=Y_0$ and $X\in \mathcal{B}_{(X_0, Y_0)}$.
Then all sets in the families $\mathcal{B}_{[W, Y]} \setminus \mathcal{B}_{[X_0, Y]}$, $\mathcal{B}_{(X_0, Y)}$, $\{Y\}$ and $\{[n]\}$ are monochromatically colored so that no two sets from different families share the same color.
We can find two elements $x \in X_0 \setminus W$ and $y \in Y \setminus X_0$. Observe that $Y \setminus \{x\} \in \mathcal{B}_{[W, Y]} \setminus \mathcal{B}_{[X_0, Y]}$ and $Y \setminus \{y\} \in \mathcal{B}_{(X_0, Y)}$. Since $c(\emptyset) = c([n])$, the sets $Y \setminus \{x\}$, $Y \setminus \{y\}$, $Y$, and $\emptyset$ all have distinct colors. Hence, the family $\{\emptyset, Y \setminus \{x\}, Y \setminus \{y\}, Y\}$ forms a rainbow induced copy of $\mathcal{B}_2$, a contradiction.
Therefore, Type 2 cannot hold in $\mathcal{B}_{[W, [n]]}$.

Assume that Type 3-1 holds in $\mathcal{B}_{[W, [n]]}$. 
\begin{itemize}
   \item Type $3$-$1$:
There exist two sets $X_0, Y_0 \in \mathcal{B}_{[W, [n]]}$ with $W < X_0 < Y_0 < [n]$ and $|Y_0| \geq |X_0| + 2$
such that
all sets in the following families are monochromatically colored
so that no two sets from different families share the same color:
$$\mathcal{B}_{[W, Y_0]} \setminus \mathcal{B}_{[X_0, Y_0]}, \quad \{X_0\}\cup \mathcal{B}^{Y_0\uparrow}_{(X_0, [n]]}\setminus \mathcal{B}_{(X_0, Y_0]}, \quad \mathcal{B}_{(X_0, Y_0)}, \quad  \{Y_0\}$$
Moreover, 
each set in $\mathcal{B}_{[W, [n]]} \setminus \left(\mathcal{B}_{[W, Y_0]} \cup \mathcal{B}^{Y_0\uparrow}_{(X_0, [n]]}\right)$
has the same color as
the sets in the first or the second family.
\end{itemize}

Note that each set in $\mathcal{B}_{[W, [n]]} \setminus \mathcal{B}_{(X_0, Y_0]}$ is colored either $c(W)$ or $c([n])$.
Since $(W, X, Y, [n])$ is a rainbow $4$-chain in 
$\mathcal{B}_{[W, [n]]}$, it follows from Type 3-1 that $X, Y\in\mathcal{B}_{(X_0, Y_0]}$. Therefore, 
 $Y=Y_0$ and $X\in \mathcal{B}_{(X_0, Y_0)}$.
The Type~3-1 case is handled identically to Type~2 with $Y=Y_0$ and $X\in\mathcal{B}_{(X_0, Y_0)}$, choosing $x\in X_0\setminus W$ and $y\in Y\setminus X_0$ again makes $\{\emptyset, Y\setminus\{x\}, Y\setminus\{y\}, Y\}$ a rainbow induced copy of $\mathcal{B}_2$, a contradiction. Therefore, Type~3-1 cannot occur in $\mathcal{B}_{[W, [n]]}$.

Assume that Type 3-2 holds in $\mathcal{B}_{[W, [n]]}$. Let $X, Y$ be two sets with $\emptyset<X<Y<[n]$.
Recall that $\mathcal{B}^{X\downarrow}_{[\emptyset, Y)}
= \bigcup_{Z \in \mathcal{B}_{(X, Y)}} \mathcal{B}_{[\emptyset, Z]}$.
Note that within $\mathcal{B}_{[W,[n]]}$, the family
$\mathcal{B}^{X_0\downarrow}_{[W, Y_0)}$ is defined in the completely analogous way: 
we replace $\emptyset$ by $W$, and the pair $(X,Y)$ by $(X_0,Y_0)$. 
Thus $\mathcal{B}^{X_0\downarrow}_{[W, Y_0)} 
= \bigcup_{Y_0' \in \mathcal{B}_{(X_0, Y_0)}} \mathcal{B}_{[W, Y_0']}.$

\begin{itemize}
    \item Type $3$-$2$:
There exist two sets $X_0, Y_0 \in \mathcal{B}_{[W, [n]]}$ with $W< X_0 < Y_0 < [n]$ and $|Y_0| \geq |X_0| + 2$
such that all sets in the following families are monochromatically colored
so that no two sets from different families share the same color:
$$
\mathcal{B}_{[X_0, [n]]} \setminus \mathcal{B}_{[X_0, Y_0]}, \quad\{Y_0\}\cup\mathcal{B}^{X_0\downarrow}_{[W, Y_0)}\setminus \mathcal{B}_{[X_0, Y_0)}, \quad  \mathcal{B}_{(X_0, Y_0)}, \quad \{X_0\}.
$$
Moreover, 
each set in $\mathcal{B}_{[W, [n]]} \setminus \left(\mathcal{B}_{[X_0, [n]]} \cup \mathcal{B}^{X_0\downarrow}_{[W, Y_0)}\right)$
has the same color as
the sets in the first or the second family. 
\end{itemize}
Note that each set in $\mathcal{B}_{[W, [n]]} \setminus \mathcal{B}_{[X_0, Y_0)}$ is colored either $c(W)$ or $c([n])$.
Since $(W, X, Y, [n])$ is a rainbow $4$-chain in 
$\mathcal{B}_{[W, [n]]}$, it follows from Type 3-2 that $X, Y\in \mathcal{B}_{[X_0, Y_0)}$, and hence
$X=X_0$ and $Y\in \mathcal{B}_{(X_0, Y_0)}$.
The Type~3-2 case is handled identically to Type~2 with $X=X_0$ and $Y\in \mathcal{B}_{(X_0, Y_0)}$, choosing $x\in X\setminus W$ and $y\in Y\setminus X$ again makes $\{\emptyset, (X\cup\{y\})\setminus\{x\}, X, Y\}$ a rainbow induced copy of $\mathcal{B}_2$, a contradiction. Therefore, Type~3-2 cannot occur in $\mathcal{B}_{[W, [n]]}$.

Assume that Type 4-1 holds in $\mathcal{B}_{[W, [n]]}$. 
\begin{itemize}
    \item Type $4$-$1$:
There exists a set $X_0$ with $W < X_0 < [n]$
and a family $\mathcal{Y}_0 \subseteq \mathcal{B}_{(X_0, [n])}$
such that 
all sets in the following families are monochromatically colored
so that no two sets from different families share the same color:
$$\bigcup_{Y' \in \mathcal{Y}_0} \left(\mathcal{B}_{[W, Y')} \setminus \mathcal{B}_{[X_0, Y')}\right), \quad \{X_0\}, \quad \mathcal{B}_{(X_0, [n]]} \setminus \mathcal{Y}_0, \quad \mathcal{Y}_0.$$
Moreover, each set in
$\mathcal{B}_{[W, [n]]} \setminus \left(\mathcal{B}_{[X_0, [n]]} \cup \bigcup_{Y'\in \mathcal{Y}_0}\mathcal{B}_{[W, Y')}\right)$
has the same color as
the sets in the first or the third family.
\end{itemize}
Note that each set in $\mathcal{B}_{[W, [n]]} \setminus (\{X_0\}\cup \mathcal{Y}_0)$ is colored either $c(W)$ or $c([n])$.
Since $(W, X, Y, [n])$ is a rainbow $4$-chain in 
$\mathcal{B}_{[W, [n]]}$, it follows from Type 4-1 that $X=X_0$ and $Y\in \mathcal{Y}_0$.
Since $Y\in\mathcal{Y}_0\subseteq \mathcal{B}_{(X, [n])}$, we can find two elements $x$ and $y$ such that $x\in X\setminus W$ and $y\in Y\setminus X$. Thus, $(X\cup\{y\})\setminus\{x\}\nsim X$ and $(X\cup\{y\})\setminus\{x\}<Y$, and hence $(X\cup\{y\})\setminus\{x\}\in \mathcal{B}_{[W, Y)} \setminus \mathcal{B}_{[X, Y)}\subseteq \bigcup_{Y' \in \mathcal{Y}_0} \left(\mathcal{B}_{[W, Y')} \setminus \mathcal{B}_{[X, Y')}\right)$.
Therefore, $c((X\cup\{y\})\setminus\{x\})=c(W)\neq c(Y)$. Note that $c(\emptyset)=c([n])$. Thus, $\emptyset, (X\cup\{y\})\setminus\{x\}, X, Y$ are with distinct
colors, and hence $\{\emptyset, (X\cup\{y\})\setminus\{x\}, X, Y\}$ forms
a rainbow induced copy of $\mathcal{B}_2$, a contradiction. Therefore, Type $4$-$1$ cannot hold.

Assume that Type 4-2 holds in $\mathcal{B}_{[W, [n]]}$. 
\begin{itemize}
    \item Type $4$-$2$:
There exists a set $Y_0$ with $W < Y_0 < [n]$
and a family $\mathcal{X}_0 \subseteq \mathcal{B}_{(W, Y_0)}$
such that 
all sets in the following families are monochromatically colored
so that no two sets from different families share the same color:
$$
\bigcup_{X' \in \mathcal{X}_0}\left(\mathcal{B}_{(X', [n]]} \setminus \mathcal{B}_{(X', Y_0]}\right), \quad 
\{Y_0\}, \quad \mathcal{B}_{[W, Y_0)} \setminus \mathcal{X}_0, \quad 
\mathcal{X}_0.
$$
Moreover, each set in
$\mathcal{B}_{[W, [n]]} \setminus \left( \mathcal{B}_{[W, Y_0]}\cup \bigcup_{X'\in \mathcal{X}_0} \mathcal{B}_{(X', [n]]}\right)$
has the same color as
the sets in the first or the third family.
\end{itemize}
Note that each set in $\mathcal{B}_{[W, [n]]} \setminus (\{Y_0\}\cup \mathcal{X}_0)$ is colored either $c(W)$ or $c([n])$.
Since $(W, X, Y, [n])$ is a rainbow $4$-chain in 
$\mathcal{B}_{[W, [n]]}$, it follows from Type 4-2 that $Y=Y_0$ and $X\in \mathcal{X}_0$.
Without loss of generality, let $c(W)=1$, $c(Z)=2$ for each $Z\in \mathcal{X}_0$, $c(Y)=3$, and $c([n])=4$. Since $c(\emptyset)=c([n])$, we have $c(\emptyset)=4$.

Note that $\mathcal{X}_0\subseteq\mathcal{B}_{(W, Y)}$.
We now claim that $\mathcal{X}_0=\mathcal{B}_{(W, Y)}$. 
Suppose to the contrary that $\mathcal{X}_0\neq\mathcal{B}_{(W, Y)}$.
Let $Z_1\in\mathcal{X}_0$.
If there exists $Z_2\in\mathcal{B}_{(W, Y)}\setminus\mathcal{X}_0$ such that
$Z_1\nsim Z_2$, then since $c(Z_1)=2$ and
$c(Z_2)=c(W)=1$, $\{\emptyset, Z_1, Z_2, Y\}$ forms a rainbow induced copy of $\mathcal{B}_2$, a contradiction.

Therefore, every $Z\in\mathcal{B}_{(W, Y)}\setminus\mathcal{X}_0$ satisfies $Z\sim Z_1$.
Now let $Z\in\mathcal{B}_{(W, Y)}\setminus\mathcal{X}_0$. Since $W< Z\sim Z_1< Y$, both sets
$(Z\cap Z_1)\setminus W$ and $Y\setminus (Z\cup Z_1)$ are nonempty. Choose $x\in Y\setminus (Z\cup Z_1)$ and $y\in (Z\cap Z_1)\setminus W$, and take
$Z'=(Z_1\cup \{x\})\setminus \{y\}$. Then $Z\not\leq Z'$ and $Z_1\not\leq Z'$ (since $y\notin Z'$), and
$Z'\not\leq Z$ and $Z'\not\leq Z_1$ (since $x\in Z'\setminus (Z\cup Z_1)$). Thus, $Z\nsim Z'$ and $Z_1\nsim Z'$. Since every $Z\in\mathcal{B}_{(W, Y)}\setminus\mathcal{X}_0$ satisfies $Z\sim Z_1$, we have $Z'\in \mathcal{X}_0$, and hence $c(Z')=2$.
As $c(Z)=1$,
$\{\emptyset, Z, Z', Y\}$ again forms a rainbow induced copy of $\mathcal{B}_2$, a contradiction.
Therefore, $\mathcal{X}_0=\mathcal{B}_{(W, Y)}$, and so
\begin{align*}
\bigcup_{X' \in \mathcal{B}_{(W, Y)}} \left(\mathcal{B}_{(X', [n]]} \setminus \mathcal{B}_{(X', Y]}\right)&=\left(\bigcup_{X' \in \mathcal{B}_{(W, Y)}} \mathcal{B}_{(X', [n]]} \right)\setminus \left(\bigcup_{X' \in \mathcal{B}_{(W, Y)}} \mathcal{B}_{(X', Y]}\right)\\[2mm]
&=\left(\bigcup_{X' \in \mathcal{B}_{(W, Y)}} \mathcal{B}_{[X', [n]]} \right)\setminus  \mathcal{B}_{(W, Y]}=\mathcal{B}^{Y\uparrow}_{(W, [n]]}\setminus \mathcal{B}_{(W, Y]},
\end{align*}
and
\begin{align*}
&\mathcal{B}_{[W, [n]]}\setminus\left(\mathcal{B}_{[W, Y]}\cup\bigcup_{X'\in\mathcal{B}_{(W, Y)}}\mathcal{B}_{(X', [n]]}\right)=\mathcal{B}_{(W, [n]]}\setminus\left(\mathcal{B}_{(W, Y]}\cup\bigcup_{X'\in\mathcal{B}_{(W, Y)}}\mathcal{B}_{(X', [n]]}\right)=\mathcal{B}_{(W, [n]]}\setminus\mathcal{B}^{Y\uparrow}_{(W, [n]]}.
\end{align*}
Consequently, all sets in the following families are monochromatically colored so that no two sets from different families
share the same color:
$$
\mathcal{B}^{Y\uparrow}_{(W, [n]]}\setminus \mathcal{B}_{(W, Y]}, \quad \{Y\}, \quad \{W\}, \quad \mathcal{B}_{(W, Y)}.
$$
Moreover, each set in
$\mathcal{B}_{(W, [n]]}\setminus\mathcal{B}^{Y\uparrow}_{(W, [n]]}$
has the same color as the sets in the first or the third family.

Then consider the colors of the sets in $\mathcal{B}_{[\emptyset, Y]}$. 
Recall that the coloring of $c$ in $\mathcal{B}_{[\emptyset, Y]}$ must be one of Types~2, 3-1, 3-2, 4-1, or 4-2, since $(\emptyset, W, X, Y)$ is a rainbow $4$-chain.
From Observation \ref{obs1}, in Types~$2$, $3$-$1$, $3$-$2$, and $4$-$2$, there exists no set $Z$ with $\mathcal{B}_{(Z, Y)}\neq \emptyset$ such that 
all sets in the families $\{Z\}$, 
$\mathcal{B}_{(Z, Y)}$, $\{Y\}$ are monochromatically colored so that no two sets from different families
share the same color. 
By the previous paragraph, we know that all sets in the families $\{W\}$, 
$\mathcal{B}_{(W, Y)}$, $\{Y\}$ are monochromatically colored so that no two sets from different families
share the same color, and hence only Type 4-1 holds with $X_0=W$ and $\mathcal{Y}_0=\mathcal{B}_{(W, Y)}$.
By a computation similar to the previous paragraph,
all sets in the following families are monochromatically colored
so that no two sets from different families share the same color:
$$\mathcal{B}^{W\downarrow}_{[\emptyset, Y)}\setminus \mathcal{B}_{[W, Y)}, \quad \{W\}, \quad \{Y\}, \quad \mathcal{B}_{(W, Y)}.$$
Moreover, each set in
$\mathcal{B}_{[\emptyset, Y)}\setminus\mathcal{B}^{W\downarrow}_{[\emptyset, Y)}$
has the same color as
the sets in the first or the third family.

Since $c(\emptyset)=c([n])$, for any rainbow $4$-chain $(W, X, Y, [n])$, all sets in
the following families are monochromatically colored
so that no two sets from different families share the same color:
$$
\left(\mathcal{B}^{Y\uparrow}_{(W, [n]]}\cup \mathcal{B}^{W\downarrow}_{[\emptyset, Y)}\right)\setminus \mathcal{B}_{[W, Y]}, \quad
\{W\}, \quad
\mathcal{B}_{(W, Y)}, \quad
\{Y\}.
$$
Moreover, each set in $\mathcal{B}_{[\emptyset, Y)}\setminus \mathcal{B}^{W\downarrow}_{[\emptyset, Y)}$
has the same color as the sets in the first or the last family, and each set in $\mathcal{B}_{(W, [n]]}\setminus \mathcal{B}^{Y\uparrow}_{(W, [n]]}$
has the same color as the sets in the first or the second family. Therefore, $(1)$ holds.

In the end, consider any rainbow $3$-chain $(W, X, [n])$ that cannot be extended to a rainbow $4$-chain by adding another set.
Suppose that there exists a set $Z_0\in \mathcal{B}_{(W, [n])}$ with
$c(Z_0)\in [k]\setminus\{c(W), c(X), c([n])\}$. Note that $W<Z_0$.
If $Z_0\nsim X$, then $\{W, Z_0, X, [n]\}$ forms a rainbow copy of $\mathcal{B}_2$, a contradiction.
If $Z_0\sim X$, then $\{W, Z_0, X, [n]\}$ forms a rainbow $4$-chain, which contradicts that $(W, X, [n])$ 
cannot be extended to a rainbow $4$-chain by adding another set. Therefore, $c(Z)\in \{c(W), c(X), c([n])\}$ for any 
$Z\in \mathcal{B}_{(W, [n])}$. Similarly, we have $c(Z)\in \{c(W), c(X), c([n])\}$ for each $Z\in \mathcal{B}_{(\emptyset, X)}$. Thus, 
the union $\mathcal{B}_{[\emptyset, X]} \cup \mathcal{B}_{[W, [n]]}$ contains $3$ colors, $(2)$ holds.
\end{proof}
\subsection{Characterization of exact $k$-colorings without a rainbow induced copy of $\mathcal{B}_2$}\label{subsection26}

We are now ready to state the main structural result of this subsection.
Consider an exact $k$-coloring $c$ of $\mathcal{B}_n$.
If $c(\emptyset)\neq c([n])$, then by Lemmas~\ref{k5} and \ref{k4}, $\mathcal{B}_n$ contains no rainbow induced copy of $\mathcal{B}_2$ under $c$ if and only if $c$ is of Type~$1$, $2$, $3$-$1$, $3$-$2$, $4$-$1$, or $4$-$2$.
If $c(\emptyset)=c([n])$, then by Lemma~\ref{emptyset=n}, $\mathcal{B}_n$ contains no rainbow induced copy of $\mathcal{B}_2$ under $c$ if and only if $c$ is of Type~$5$.
Therefore, we obtain the following theorem.

\begin{theorem}\label{Structural-B2-1}
Let $k, n$ be two integers with $k\geq 4$ and $n>4$. In an exact $k$-coloring $c$ of $\mathcal{B}_n$, there is no rainbow induced copy of $\mathcal{B}_2$ if and only if one of the following holds:
\begin{itemize}
\item[$(1)$] $c(\emptyset)\neq c([n])$, and the coloring $c$ is of one of Types $1$, $2$, $3$-$1$, $3$-$2$, $4$-$1$, or $4$-$2$.
\item[$(2)$] $c(\emptyset)=c([n])$, and the coloring $c$ is of Type~$5$.
\end{itemize}
\end{theorem}

To obtain an upper bound on the number of colors in an exact coloring of $\mathcal{B}_n$ with no rainbow induced copy of $\mathcal{B}_2$, we introduce some auxiliary notation and use a result of Griggs, Stahl, and Trotter.
For positive integers $u$ and $v$, let $\mathcal{C}_{u, v}$ be a union of $u$ disjoint $v$-chains, where sets from different chains are pairwise incomparable. More precisely, let the family of sets be denoted by
$\{X_{i, j}  :  1 \leq i \leq u,\; 1 \leq j \leq v\}$,
satisfying
$X_{i, 1} < X_{i, 2} < \cdots < X_{i, v}$ for each $i$,
and
$X_{i, j} \nsim X_{i', j'}$ for any $i \ne i'$ and any $j, j'$.
For given $v$ and $n$, let $c_v(n)$ be the maximum integer $u$ such that there exists a copy of $\mathcal{C}_{u, v}$ in $\mathcal{B}_n$.
Griggs, Stahl, and Trotter~\cite{GST84} determined $c_v(n)$ exactly.

\begin{theorem}[Griggs--Stahl--Trotter {\cite{GST84}}]\label{cvn} 
Let $v,n$ be positive integers. Then
$c_v(n) = {n - v + 1 \choose \left\lfloor (n - v + 1)/2 \right\rfloor}$.
\end{theorem}

As an application, we obtain the following upper bound.

\begin{proposition}\label{empty-n-k}
Let $k, n$ be two integers with $n>2$ and $k\geq 4$. Consider an exact $k$-coloring $c$ of $\mathcal{B}_n$ with no rainbow induced copy of $\mathcal{B}_2$. Then
    $k \leq \binom{n}{\lfloor n/2\rfloor} + \binom{n-1}{\lfloor (n-1)/2\rfloor} + \binom{n-2}{\lfloor (n-2)/2\rfloor} + 1, $
\end{proposition}
\begin{proof}
When $c(\emptyset)\neq c([n])$, by Theorem \ref{Structural-B2-1} (1), the coloring $c$ is of one of Types $1$, $2$, $3$-$1$, $3$-$2$, $4$-$1$, or $4$-$2$. Thus, $k\leq 5$. Considering the case $c(\emptyset)= c([n])$, by Theorem \ref{Structural-B2-1} (2), the coloring $c$ is of Types $5$.
For the exact $k$-coloring of $\mathcal{B}_n$, we can find a rainbow family $\{X_i  :  1 \leq i \leq k - 1\}$ in $\mathcal{B}_{(\emptyset, [n])}$ such that $c(X_i)\neq c([n])$ and $c(X_i) \neq c(X_j)$ for $i, j$ with $1\leq i\neq j\leq k-1$. By Lemma~\ref{clm2}, every rainbow chain in $\mathcal{B}_n$ has length at most $4$.
Therefore, every rainbow chain in $\{X_i  :  1 \leq i \leq k - 1\}$ has length at most $3$.

We construct a decomposition of $\{X_i : 1 \leq i \leq k-1\}$ as follows.
First, extract a maximal collection of pairwise disjoint rainbow $3$-chains and denote them by $\mathcal{C}_3^1, \dots, \mathcal{C}_3^a$; remove all their sets. 
Next, from the remaining sets extract a maximal collection of pairwise disjoint rainbow $2$-chains, denoted $\mathcal{C}_2^1, \dots, \mathcal{C}_2^b$, and remove these as well. 
Finally, let $\mathcal{A}_c$ be the family of all sets that remain; then $\mathcal{A}_c$ is an antichain of size $c$. 
Define $\mathcal{C}(a, b, c)=\{\mathcal{C}_3^1, \dots, \mathcal{C}_3^a, \mathcal{C}_2^1, \dots, \mathcal{C}_2^b, \mathcal{A}_c\}$. 

Next, we show that any two sets belonging to different chains or antichains in $\mathcal{C}(a, b, c)$ are incomparable.
We first prove that $X \nsim Y$ for each $i$ with $1 \leq i \leq a$ and each $X \in \mathcal{C}_3^i$, $Y \in \mathcal{C}(a, b, c) \setminus \mathcal{C}_3^i$. Suppose, for contradiction, that for some $i$ there exist $X \in \mathcal{C}_3^i$ and $Y \in \mathcal{C}(a, b, c) \setminus \mathcal{C}_3^i$ such that $Y \sim X$. 
Assume that $Y>X$.
Suppose that there exists $X'\in \mathcal{C}_3^i$ with $X'\nsim Y$.
If $X'<X$, then $X'<X<Y$, contradicting $X'\nsim Y$. Thus, $X'>X$, and hence $\{X, X', Y, [n]\}$ forms a rainbow induced copy of $\mathcal{B}_2$, a contradiction.
{Therefore, $Y \sim X'$ for all $X' \in \mathcal{C}_3^i$, and hence $Y$ together with $\mathcal{C}_3^i$ forms a rainbow 4-chain, contradicting the fact that rainbow chains in $\mathcal{B}_{(\emptyset, [n])}$ without color $c([n])$ have size at most $3$.}
Assume that $Y<X$. {Similarly, we can also obtain a contradiction.} Therefore, such $Y$ cannot exist.

By the same argument in the previous paragraph, we can show that $X \nsim Y$ for each $i$ with $1 \leq i \leq b$, each $X \in \mathcal{C}_2^i$, and each $Y \in \mathcal{A}_c\cup \bigcup_{j=1}^b \mathcal{C}_2^j \setminus \mathcal{C}_2^i$. Therefore, any two sets belonging to different chains or antichains in $\mathcal{C}(a, b, c)$ are incomparable.

Thus, $\mathcal{C}(a, b, c)$ contains subposets isomorphic to $\mathcal{C}_{a+b+c, 1}$, $\mathcal{C}_{a+b, 2}$, and $\mathcal{C}_{a, 3}$. 
Recall that $\mathcal{C}_{u, v}$ is a union of $u$ disjoint $v $-chains, where sets from different chains are pairwise incomparable. 
By Theorem~\ref{cvn}, we have
$$
a + b + c \leq \binom{n}{\lfloor n/2 \rfloor}, \quad
a + b \leq \binom{n-1}{\lfloor (n-1)/2 \rfloor}, \quad
a \leq \binom{n-2}{\lfloor (n-2)/2 \rfloor}.
$$
Since $|\mathcal{C}(a, b, c)| = k - 1 = 3a + 2b + c$, it follows that
$k \leq \binom{n}{\lfloor n/2 \rfloor} + \binom{n-1}{\lfloor (n-1)/2 \rfloor} + \binom{n-2}{\lfloor (n-2)/2 \rfloor} + 1.$
\end{proof}

%%%%%%%%%%%%%%%%%%%%%%%%%%%%%%%%%%%%%%%%%%%%%%%%%%%%%%%%%%%%%%%%%%%%%%%%%%%
\section{Application to Gallai-Ramsey theory}\label{sec_3}
%%%%%%%%%%%%%%%%%%%%%%%%%%%%%%%%%%%%%%%%%%%%%%%%%%%%%%%%%%%%%%%%%%%%%%%%%%%
\newtheorem*{mainthm3}{\rm\bf Theorem~\ref{cBGR-c3}}
\begin{mainthm3}[Restated]
Let $k, s$ be integers with $k\geq 3$ and $s\geq 3$. Then 
\begin{itemize}
    \item[] $(1)$ $\mathrm{GR}_{k}(\mathcal{C}_3:\mathcal{C}_s)=s$, if $3\leq k\leq{s-1\choose \lceil (s-1)/2\rceil}+1$.
    
    \item[] $(2)$ The pair of posets $\mathcal{C}_3, \mathcal{C}_s$ is $(\mathcal{C}_3:\mathcal{C}_s)_k$-good if $k>{s-1\choose \lceil (s-1)/2\rceil}+1$.
\end{itemize}
\end{mainthm3}
\begin{proof}
To prove (1) and (2),  
we first show Claim \ref{claim-1}.
\begin{claim}\label{claim-1}
There exists either a rainbow induced copy of $\mathcal{C}_3$ or a monochromatic induced copy of $\mathcal{C}_N$  for any $N\geq 3$ and any coloring of the sets in $\mathcal{B}_N$ with at least $3$ colors.
\end{claim}
\begin{proof}
We proceed by induction on $N$.
When $N=3$, it is easy to show that for any coloring of the sets in $\mathcal{B}_N$ with at least $3$ colors, there exists either a rainbow induced copy of $\mathcal{C}_3$ or a monochromatic induced copy of $\mathcal{C}_3$.
Thus, we assume that for $N\geq 4$ and any coloring of the sets in $\mathcal{B}_{N-1}$ with at least $3$ colors, there exists either a rainbow induced copy of $\mathcal{C}_3$ or a monochromatic induced copy of $\mathcal{C}_{N-1}$.
Consider a coloring $c$ of the sets in $\mathcal{B}_N$ with at least $3$ colors. Suppose, toward a contradiction, that there exists neither a rainbow induced copy of $\mathcal{C}_3$ nor a monochromatic induced copy of $\mathcal{C}_N$.

Without loss of generality, let $c(\emptyset)=1$. 
By Theorem \ref{Structural-C3-1} (1), 
we can obtain that $c(\emptyset)=c([N])=1$.
Then consider the restrictions of $c$ to $\mathcal{B}_{[\emptyset,[N-1]]}$ and to $\mathcal{B}_{[\{N\},[N]]}$. 
Suppose first that at least $3$ colors appear on the sets in $\mathcal{B}_{[\emptyset,[N-1]]}$; the case where at least $3$ colors appear on the sets in $\mathcal{B}_{[\{N\},[N]]}$ is analogous.
By the induction hypothesis applied to the restriction of $c$ to $\mathcal{B}_{[\emptyset,[N-1]]}$, there exists either a rainbow induced copy of $\mathcal{C}_3$ or a monochromatic induced copy of $\mathcal{C}_{N-1}$ in $\mathcal{B}_{[\emptyset,[N-1]]}$.

If there is a rainbow induced copy of $\mathcal{C}_3$ in $\mathcal{B}_{[\emptyset,[N-1]]}$, then it is also a rainbow induced copy of $\mathcal{C}_3$ in $\mathcal{B}_N$, a contradiction. Therefore, there is a monochromatic induced copy of $\mathcal{C}_{N-1}$ in $\mathcal{B}_{[\emptyset,[N-1]]}$.
Since $\mathcal{B}_{(\emptyset, [N-1])}$ contains no induced copy of $\mathcal{C}_{N-1}$, this chain must include $\emptyset$ or $[N-1]$. 
By Theorem \ref{Structural-C3-1} (1), 
$c([N-1])=c(\emptyset)=1$. 
Consequently, there exists a monochromatic induced copy of $\mathcal{C}_{N-1}$ of color $1$ in $\mathcal{B}_{[\emptyset, [N-1]]}$.
Since $c([N])=1$, it follows that there exists a monochromatic induced copy of $\mathcal{C}_{N}$ with color $1$ in $\mathcal{B}_{[\emptyset, [N]]}$, a contradiction.
Therefore, each of $\mathcal{B}_{[\emptyset,[N-1]]}$ and $\mathcal{B}_{[\{N\},[N]]}$ uses at most two colors.

If there is only one color in $\mathcal{B}_{[\emptyset, [N-1]]}$ or $\mathcal{B}_{[\{N\}, [N]]}$, then there exists a monochromatic induced copy of $\mathcal{C}_{N}$, a contradiction.
Hence, there exist $2$ colors in each of $\mathcal{B}_{[\emptyset, [N-1]]}$ and $\mathcal{B}_{[\{N\}, [N]]}$. Note that there exist at least $3$ colors in $\mathcal{B}_N$ and $c(\emptyset)=c([N])=1$.
We assume $c(X)\in\{1, 2\}$ for each $X\in \mathcal{B}_{[\emptyset, [N-1]]}$ and $c(Y)\in\{1, 3\}$ for each $Y\in \mathcal{B}_{[\{N\}, [N]]}$.

Let $\mathcal{X}$ be the family of sets in $\mathcal{B}_{[\emptyset, [N-1]]}$ with color $2$. Note that $\mathcal{X}\neq \emptyset$. Then we can find a set $X_0$ satisfying $|X_0|=\min\{|X| :  X\in\mathcal{X}\}$. Thus, $c(X)=1$ for each $X\in \mathcal{B}_{[\emptyset, X_0)}$.
By Theorem \ref{Structural-C3-1}, we see that $c(Y)=1$ for each $Y\in \mathcal{B}_{[X_0\cup\{N\}, [N])}$.
Note that $c([N])=1$. 
Thus, there exists a monochromatic induced copy of $\mathcal{C}_N$ in $\mathcal{B}_{[\emptyset, X_0)} \cup \mathcal{B}_{[X_0 \cup \{N\}, [N]]}$ with color $1$, a contradiction. 
Thus, the claim holds.
\end{proof}

When $3\leq k\leq{s-1\choose \lceil (s-1)/2\rceil}+1$,
by Claim \ref{claim-1}, there exists a rainbow induced copy of $\mathcal{C}_3$ or a monochromatic induced copy of $\mathcal{C}_{s}$ for any $N\geq s$ and any exact $k$-coloring of the sets in $\mathcal{B}_{N}$. Therefore, 
$\mathrm{GR}_{k}(\mathcal{C}_3:\mathcal{C}_s)\leq s$.

To show that $\mathrm{GR}_{k}(\mathcal{C}_3:\mathcal{C}_s)\geq s$, where $3\leq k\leq {s-1\choose \lceil (s-1)/2\rceil}+1$, we give an exact $k$-coloring $c'$ of $\mathcal{B}_{s-1}$ as follows.  First, we can find $k-1$ different sets $X_1, X_2, \ldots, X_{k-1}$ in ${[s-1] \choose \lfloor(s-1)/2\rfloor}$ and let $c'(X_i)=i+1$ for $1\leq i\leq k-1$, $c'(X)\in [2, k]$ for each $X\in {[s-1] \choose \lfloor(s-1)/2\rfloor}\setminus \{X_i : 1\leq i\leq k-1\}$. Then  let $c'(Y)=1$ for  $Y\in \mathcal{B}_{s-1}\setminus {[s-1] \choose \lfloor(s-1)/2\rfloor}$. 
It is clear that there exists neither a rainbow induced copy of $\mathcal{C}_3$ nor a monochromatic induced copy of $\mathcal{C}_s$ in the exact $k$-coloring $c'$ of $\mathcal{B}_{s-1}$, and hence $\mathrm{GR}_{k}(\mathcal{C}_3:\mathcal{C}_s)= s$.

When $k> {s-1\choose \lceil (s-1)/2\rceil}+1$, we show that the pair of posets $\mathcal{C}_3$ and $\mathcal{C}_s$ is $(\mathcal{C}_3 : \mathcal{C}_s)_k $-good.
Since $k> {s-1\choose \lceil (s-1)/2\rceil}+1\geq 3$, it follows from Claim \ref{claim-1} that for any $N$ with $N\geq s$ and any exact $k$-coloring of the sets in $\mathcal{B}_{N}$, there exists a rainbow induced copy of $\mathcal{C}_3$ or a monochromatic induced copy of $\mathcal{C}_{s}$.

Then, consider each $N \leq s-1$. Since 
$k > \binom{s-1}{\lceil (s-1)/2 \rceil} + 1 \ge \binom{N}{\lceil N/2 \rceil} + 1$,
it follows from Corollary~\ref{cor:upper-bound-k} that every exact $k$-coloring of $\mathcal{B}_N$
contains a rainbow induced copy of $\mathcal{C}_3$.
Therefore, the pair of posets $\mathcal{C}_3$ and $\mathcal{C}_s$ is $(\mathcal{C}_3 : \mathcal{C}_s)_k $-good for all
$k > \binom{s-1}{\lceil (s-1)/2 \rceil} + 1.$
\end{proof}

\newtheorem*{mainthm4}{\rm\bf Theorem~\ref{cBGR-v2}}
\begin{mainthm4}[Restated]
Let $k, s$ be two integers with $k\geq 3$ and $s\geq 2$. Then 
\begin{itemize}
    \item[] $(1)$ $\mathrm{GR}_{3}(\vee_2:\mathcal{C}_s)=2s-1$.
     \item[] $(2)$ {The pair of posets \(\vee_2\) and \(\mathcal{C}_s\) is \((\vee_2:\mathcal{C}_s)_4\)-good for \(s\in\{2,3\}\), and \(\operatorname{GR}_{4}(\vee_2:\mathcal{C}_s)=s\) for all \(s\ge 4\).}  
    \item[] $(3)$ The pair of $\vee_2, \mathcal{C}_s$ is $(\vee_2:\mathcal{C}_s)_k$-good if $k\geq 5$.
\end{itemize}
\end{mainthm4}
\begin{proof}
$(1)$ For the lower bound, we can give an exact $3$-coloring $c_0$ of $\mathcal{B}_{2s-2}$ such that there exists neither a rainbow induced copy of $\vee_2$ nor a monochromatic induced copy of $\mathcal{C}_s$, as follows. Let $c_0(X)=1$ for each $X$ with $0\leq |X|\leq s-2$, $c_0(X)=2$ for each $X$ with $s-1\leq |X|\leq 2s-3$, and $c_0([2s-2])=3$. Then there is neither a rainbow induced copy of $\vee_2$ nor a monochromatic induced copy of $\mathcal{C}_s$.

For the upper bound, we suppose that there exists neither a rainbow induced copy of $\vee_2$ nor a monochromatic induced copy of $\mathcal{C}_s$ in some exact $3$-coloring $c$ of $\mathcal{B}_{N}$ for some $N\geq 2s-1$. Since there exists no rainbow induced copy of $\vee_2$, it follows from Theorem \ref{Structural-V2-1} that 
   one of the following conditions holds:
 \begin{itemize}
    \item[$(i)$] There exists a set $A$ with $|A|\leq n-2$ such that
    all sets in the families $\mathcal{B}_{(A, [N])}, \mathcal{B}_N\setminus \mathcal{B}_{[A, [N]]},\{A\}$ are monochromatically colored so that no two sets from different families share the same color. Moreover, the color of $[N]$ is unrestricted.
    \item[$(ii)$] Exactly two colors appear in $\mathcal{B}_{[\emptyset, [N])}$. The set $[N]$ receive a color, distinct from the two used in $\mathcal{B}_{[\emptyset, [N])}$.
\end{itemize}
When $(i)$ holds, we may fix such a set $A$ and find an element $a\in A$.
Thus $a\in X$ for each $X\in \mathcal{B}_{[A, [N]]}$ and $a\not\in X$ for each $X\in \mathcal{B}_{[\emptyset, [N]\setminus\{a\}]}$, and hence $\mathcal{B}_{[A, [N]]}\cap \mathcal{B}_{[\emptyset, [N]\setminus\{a\}]}=\emptyset$.
Then all sets in $\mathcal{B}_{[\emptyset, [N]\setminus\{a\}]}$ are monochromatically colored. Therefore, $\mathcal{B}_{[\emptyset, [N]\setminus\{a\}]}$ contains a monochromatic induced copy of $\mathcal{C}_s$, a contradiction. 
When $(ii)$ holds, there exists a $(2s-1)$-chain with at most $2$ colors, and hence there is a monochromatic induced copy of $\mathcal{C}_s$, a contradiction.  
Therefore,  $\mathrm{GR}_{k}(\vee_2:\mathcal{C}_s)=2s-1$.

(2) 
{It is easy to see that the pair of posets \(\vee_2\) and \(\mathcal{C}_s\) is \((\vee_2:\mathcal{C}_s)_4\)-good for \(s\in\{2,3\}\). Then, we consider the case $s\geq 4$.}
For the lower bound, we can give an exact $k$-coloring $c_0$ of $\mathcal{B}_{s-1}$ such that there exists neither a rainbow induced copy of $\vee_2$ nor a monochromatic induced copy of $\mathcal{C}_s$, as follows. Let $c_0(\{1\})=2$, $c_0(X)=3$ for each $X\in \mathcal{B}_{(\{1\}, [s-1])}$, $c_0(X)=1$ for each $X\in \mathcal{B}_{s-1}\setminus
 \mathcal{B}_{[\{1\}, [s-1]]}$, and $c_0([s-1])=4$. Note that $s\geq 4$ implies $\mathcal{B}_{(\{1\}, [s-1])}\neq \emptyset$. Then there is neither a rainbow induced copy of $\vee_2$ nor a monochromatic induced copy of $\mathcal{C}_s$.
 
For the upper bound, we suppose that there exists neither a rainbow induced copy of $\vee_2$ nor a monochromatic induced copy of $\mathcal{C}_s$ in some exact $4$-coloring $c$ of $\mathcal{B}_{N}$ for some $N\geq s$. Since there exists no rainbow induced copy of $\vee_2$, it follows from Theorem \ref{Structural-V2-1} that 
 there exists a set $A$ such that all sets in the families $\mathcal{B}_{(A, [N])}, \mathcal{B}_n\setminus \mathcal{B}_{[A, [N]]},\{A\}$ are monochromatically colored so that no two sets from different families share the same color.
By the same argument in $(1)$, we can find an element
$a\in A$ such that $\mathcal{B}_{[\emptyset, [N]\setminus\{a\}]}$ contains a monochromatic induced copy of $\mathcal{C}_s$, a contradiction. Therefore,  $\mathrm{GR}_{k}(\vee_2:\mathcal{C}_s)=s$.

(3) Since $k\geq 5$, it follows from Theorem \ref{Structural-V2-1} that for any $N$ and any exact $k$-coloring of the sets in $\mathcal{B}_{N}$, there exists a rainbow induced copy of $\vee_2$, and hence the pair of posets $\vee_2$ and $\mathcal{C}_s$ is $(\vee_2:\mathcal{C}_s)_k$-good.
\end{proof}

\newtheorem*{mainthm5}{\rm\bf Theorem~\ref{gr2m}}
\begin{mainthm5}[Restated]
Let $k, n$ be two integers with $k\geq 4$ and $n\geq 1$. 
\begin{itemize}
    \item[] $(1)$ If $4\leq k\leq 2^{\mathrm{R}_3\left(\mathcal{B}_n\right)+n}$, then $\mathrm{GR}_k(\mathcal{B}_2:\mathcal{B}_n) \leq \mathrm{R}_3\left(\mathcal{B}_n\right)+n$.
    \item[] $(2)$ If $k>2^{\mathrm{R}_3\left(\mathcal{B}_n\right)+n}$, then the pair of posets $\mathcal{B}_2, \mathcal{B}_n$ is $(\mathcal{B}_2:\mathcal{B}_n)_k$-good.
\end{itemize} 
\end{mainthm5}
\begin{proof}
(1) Since $4\leq k\leq 2^{\mathrm{R}_3\left(\mathcal{B}_n\right)+n}$, there always exist exact $k$-coloring of $\mathcal{B}_n$ for any $N\ge \mathrm{R}_3(\mathcal{B}_n)+n$.
To prove $\mathrm{GR}_k(\mathcal{B}_2:\mathcal{B}_n)\le \mathrm{R}_3(\mathcal{B}_n)+n$, assume for contradiction that for some $N\ge \mathrm{R}_3(\mathcal{B}_n)+n$ there exists an exact $k$-coloring $c$ of $\mathcal{B}_N$ which contains neither a rainbow induced copy of $\mathcal{B}_2$ nor a monochromatic induced copy of $\mathcal{B}_n$.

Suppose that $c(\emptyset)=c([N])$. By Theorem~\ref{Structural-B2-1}~(2), the coloring $c$ is of Type~5. Thus, the maximum length of any rainbow chain in $\mathcal{B}_N$ is at most $4$. From Proposition \ref{empty-n-k}, we have
$$
k \le \binom{N}{\lfloor N/2\rfloor}+\binom{N-1}{\lfloor (N-1)/2\rfloor}+\binom{N-2}{\lfloor (N-2)/2\rfloor}+1.
$$

If there exists a rainbow $4$-chain $(W, X, Y, [N])$  in $\mathcal{B}_N$, it follows from Type~5 (1) that all sets in the families $\mathcal{B}_{[\emptyset, W)}$, $\mathcal{B}_{(W, Y)}$, and $\mathcal{B}_{(Y, [N]]}$ are monochromatically colored.
It is easy to see that $\mathrm{R}_3(\mathcal{B}_n)\geq 3n$. Thus $N\geq 4n$, and hence one of $\mathcal{B}_{[\emptyset, W)}$, $\mathcal{B}_{(W, Y)}$, and $\mathcal{B}_{(Y, [N]]}$ contains a monochromatic induced copy of $\mathcal{B}_n$, a contradiction.

Therefore, the maximum length of any rainbow chain in $\mathcal{B}_N$ is at most $3$.
Let $W$ be a set with $c(W)\neq c(\emptyset)$ such that each set in $\mathcal{B}_{[\emptyset, W)}$ is colored $c(\emptyset)$.
Then from Type 5 (2), there exist at most $3$ colors in $\mathcal{B}_{[W, [N]]}$.
If $|W|\ge n$, then $\mathcal{B}_{[\emptyset, W)}\cup\{[N]\}$ contains a monochromatic induced copy of $\mathcal{B}_n$, a contradiction. If $|W|< n$, i.e., $N-|W|\ge \mathrm{R}_3(\mathcal{B}_n)$, then there exists a monochromatic induced copy of $\mathcal{B}_n$ in $\mathcal{B}_{[W, [N]]}$, also a contradiction.

Suppose that $c(\emptyset)\neq c([N])$, then Theorem~\ref{Structural-B2-1}~(1) holds. In Types 1, 2, 3-1, 3-2, and 4-1, there exists a set $X_0$ such that there exist at most $3$ colors in $\mathcal{B}_N\setminus \mathcal{B}_{[X_0, [N]]}$. Note that $X_0\neq \emptyset$, and hence  $\mathcal{B}_N\setminus \mathcal{B}_{[X_0, [N]]}$ contains an induced copy of $\mathcal{B}_{N-1}$. 
Since $N-1\geq \mathrm{R}_3(\mathcal{B}_n)$, $\mathcal{B}_N\setminus \mathcal{B}_{[X_0, [N]]}$ contains a monochromatic induced copy of $\mathcal{B}_n$, a contradiction.

 In Type 4-2, there exists a set $Y_0$ such that there exist at most $3$ colors in $\mathcal{B}_N\setminus \mathcal{B}_{[\emptyset, Y_0]}$. Similarly, $\mathcal{B}_N\setminus \mathcal{B}_{[\emptyset, Y_0]}$ contains a monochromatic induced copy of $\mathcal{B}_n$, a contradiction. 
Hence no such coloring $c$ exists, proving $\mathrm{GR}_k(\mathcal{B}_2:\mathcal{B}_n)\le \mathrm{R}_3(\mathcal{B}_n)+n$.

(2) 
For any positive intger $N$, consider exact $k$-coloring of $\mathcal{B}_N$. Since $k>2^{\mathrm{R}_3\left(\mathcal{B}_n\right)+n}$, we have $N\geq \mathrm{R}_3(\mathcal{B}_n)+n$. By the same argument in (1), we can prove that for any $N$ and any exact $k$-coloring of the sets in $\mathcal{B}_{N}$, there exists either a rainbow induced copy of $\mathcal{B}_2$ or a monochromatic induced copy of $\mathcal{B}_n$, and hence the pair of posets $\mathcal{B}_2, \mathcal{B}_n$ is $(\mathcal{B}_2:\mathcal{B}_n)_k$-good.
\end{proof}

\section{Application to rainbow Ramsey theory}\label{sec_4}

We are now in a position to give an upper bound for $\operatorname{RR}(\mathcal{B}_m:\mathcal{B}_{n})$.
 \newtheorem*{mainthm6}{\rm\bf Theorem~\ref{rainbow-Boolean-m-n}}
\begin{mainthm6}[Restated]\label{thm:RR-upper}
Let \(m,n\) be positive integers. Then
\(\operatorname{RR}(\mathcal{B}_m:\mathcal{B}_{n})\le m\operatorname{R}_{2^m-1}(\mathcal{B}_{n})+m.\)
\end{mainthm6}

We first provide a sufficient condition for a coloring of \(\mathcal{B}_N\) without rainbow induced copy of \(\mathcal{B}_m\).

\begin{lemma}\label{lem:no-rainbow-Bm-structure}
Let \(m\) and \(n_0\) be two positive integers. Let \(N=mn_0+m\), and define
\(R_{\ell}=[m+(\ell-1)n_0+1,\, m+\ell n_0]\) for \(1\le \ell\le m\).
For \(1\le j\le m\) and \(1\le i_1<\cdots<i_j\le m\), define the Boolean sublattice
\[
X_{i_1,i_2,\ldots,i_j}
=\mathcal{B}_{\bigl[\{i_1,i_2,\ldots,i_j\}\cup R_1\cup\cdots\cup R_{j-1},\;
\{i_1,i_2,\ldots,i_j\}\cup R_1\cup\cdots\cup R_{j}\bigr]}
\subseteq \mathcal{B}_N .
\]
For any coloring \(c\) of \(\mathcal{B}_N\), if \(\mathcal{B}_N\) contains no rainbow induced copy of \(\mathcal{B}_m\), then there exist \(j\) and \(1\le i_1<\cdots<i_j\le m\) such that the restriction of \(c\) to \(X_{i_1,i_2,\ldots,i_j}\) uses at most \(2^m-1\) colors.
\end{lemma}

\begin{proof}
Suppose, to the contrary, that for every \(j\) and every \(1\le i_1<\cdots<i_j\le m\), the restriction of \(c\) to \(X_{i_1,i_2,\ldots,i_j}\) uses at least \(2^m\) colors.
First note that each \(X_{i_1,i_2,\ldots,i_j}\) is an induced copy of \(\mathcal{B}_{n_0}\), and there are \(2^m-1\) such sublattices.

Next, for two different sublattices \(X_{i_1,\ldots,i_j}\) and \(X_{i'_1,\ldots,i'_{j'}}\), we have
\(X\subseteq Y\) for any \(X\in X_{i_1,\ldots,i_j}\) and any \(Y\in X_{i'_1,\ldots,i'_{j'}}\) if and only if
\(\{i_1,\ldots,i_j\}\subseteq \{i'_1,\ldots,i'_{j'}\}\).
Therefore, by taking \(\emptyset\) together with one set selected from each sublattice \(X_{i_1,\ldots,i_j}\), we obtain an induced copy of \(\mathcal{B}_m\).
Moreover, since each \(X_{i_1,\ldots,i_j}\) contains at least \(2^m\) colors, we can choose these \(2^m\) sets so that their colors are pairwise distinct, which yields a rainbow induced copy of \(\mathcal{B}_m\), a contradiction.
This contradiction proves the lemma.
\end{proof}

We now apply Lemma~\ref{lem:no-rainbow-Bm-structure} to prove Theorem~\ref{thm:RR-upper}.
\begin{proof}[Proof of Theorem \ref{thm:RR-upper}]
Let \(n_0=\operatorname{R}_{2^m-1}(\mathcal{B}_{n})\) and \(N=mn_0+m\).
Consider an arbitrary coloring \(c\) of the sets in \(\mathcal{B}_N\).
We claim that \(c\) contains either a monochromatic induced copy of \(\mathcal{B}_n\) or a rainbow induced copy of \(\mathcal{B}_m\).

Define \(R_{\ell}=[m+(\ell-1)n_0+1,\, m+\ell n_0]\) for \(1\le \ell\le m\), and for
\(1\le j\le m\) and \(1\le i_1<\cdots<i_j\le m\) let
\[
X_{i_1, i_2, \ldots, i_j}
=\mathcal{B}_{\bigl[\{i_1, i_2, \ldots, i_j\}\cup R_1\cup \cdots \cup R_{j-1},\;
\{i_1, i_2, \ldots, i_j\}\cup R_1\cup \cdots \cup R_j\bigr]}.
\]
As in Lemma~\ref{lem:no-rainbow-Bm-structure}, each \(X_{i_1,\ldots,i_j}\) is an induced copy of \(\mathcal{B}_{n_0}\), and there are \(2^m-1\) such sublattices.

If \(\mathcal{B}_N\) contains a rainbow induced copy of \(\mathcal{B}_m\), then we are done.
Otherwise, by Lemma~\ref{lem:no-rainbow-Bm-structure} there exist \(j\) and
\(1\le i_1<\cdots<i_j\le m\) such that the restriction of \(c\) to
\(X_{i_1,\ldots,i_j}\) uses at most \(2^m-1\) colors.
Since \(n_0=\operatorname{R}_{2^m-1}(\mathcal{B}_{n})\), it follows that
\(X_{i_1,\ldots,i_j}\) contains a monochromatic induced copy of \(\mathcal{B}_n\).
Therefore \(\operatorname{RR}(\mathcal{B}_m:\mathcal{B}_{n})\le N
= m\operatorname{R}_{2^m-1}(\mathcal{B}_{n})+m\), completing the proof.
\end{proof}

\newtheorem*{mainthm7}{\rm\bf Theorem~\ref{th-Bm2}}
\begin{mainthm7}[Restated]
For the Boolean lattice $\mathcal{B}_n$ with $n\geq 1$, we have
$$
\mathrm{R}_3\left(\mathcal{B}_n\right)\leq \mathrm{RR}(\mathcal{B}_2:\mathcal{B}_n) \leq \mathrm{R}_3\left(\mathcal{B}_n\right)+n.
$$   
\end{mainthm7}
\begin{proof}
It is clear that $\mathrm{RR}(\mathcal{B}_2:\mathcal{B}_n)\geq \mathrm{R}_3\left(\mathcal{B}_n\right)$.
Let $N=\mathrm{R}_3\left(\mathcal{B}_n\right)+n$.
If at most $3$ colors are used in $\mathcal{B}_N$, 
it is clear that there exists a monochromatic induced copy of $\mathcal{B}_n$.
If at least $4$ colors are used in $\mathcal{B}_N$, then by Theorem
\ref{gr2m}, there exists either a rainbow induced copy of $\mathcal{B}_2$ or a monochromatic induced copy of $\mathcal{B}_n$.
Therefore, $\mathrm{RR}(\mathcal{B}_2:\mathcal{B}_n) \leq \mathrm{R}_3\left(\mathcal{B}_n\right)+n$.
\end{proof}

As observed by Chang, Gerbner, Li, Methuku, Nagy, Patk\'os, and Vizer \cite{CGLMNPV22}, 
if $\mathcal{P}$ is uniformly induced Lubell-bounded, then ${\rm R}_k(\mathcal{P})=k \cdot e(\mathcal{P})$ and $\operatorname{RR}(\mathcal{Q}:\mathcal{P}) \ge e(\mathcal{P})(|\mathcal{Q}| - 1) + g(\mathcal{Q})$ for each poset $\mathcal{Q}$. Note that $\mathrm{RR}(\vee_2:\mathcal{C}_1)=0$.
\newtheorem*{mainthm8}{\rm\bf Theorem~\ref{v2-p}}
\begin{mainthm8}[Restated]
Let $\mathcal{P}$ be a uniformly induced Lubell-bounded poset, other than $\mathcal{C}_1$. Then $\mathrm{RR}(\vee_2:\mathcal{P})=2e(\mathcal{P})+1$.
\end{mainthm8}
\begin{proof}
Recall that $\operatorname{RR}(\mathcal{Q}:\mathcal{P}) \ge e(\mathcal{P})(|\mathcal{Q}| - 1) + g(\mathcal{Q})$ for each poset $\mathcal{Q}$. Therefore, $\mathrm{RR}(\vee_2:\mathcal{P})\geq 2e(\mathcal{P})+1$.

To show $\mathrm{RR}(\vee_2:\mathcal{P})\leq 2e(\mathcal{P})+1$, we suppose that there exists neither a rainbow induced copy of $\vee_2$ nor a monochromatic induced copy of $\mathcal{P}$ in some coloring $c$ of $\mathcal{B}_{N}$, where $N=2e(\mathcal{P})+1$. 
If at most $2$ colors are used to color the sets in $\mathcal{B}_{N}$, then since $N>{\rm R}_2(\mathcal{P})=2e(\mathcal{P})$, there is a monochromatic induced copy of $\mathcal{P}$, a contradiction. Therefore, at least $3$ colors are used in the coloring of $\mathcal{B}_{N}$.

Since there exists no rainbow induced copy of $\vee_2$, it follows from Theorem \ref{Structural-V2-1} that 
one of the following conditions holds:
 \begin{itemize}
    \item[$(1)$] There exists a set $A$ such that
    all sets in the families $\mathcal{B}_{(A, [N])}, \mathcal{B}_n\setminus \mathcal{B}_{[A, [N]]},\{A\}$ are monochromatically colored so that no two sets from different families share the same color.
If $k=3$, then the color of $[N]$ is unrestricted; if $k=4$, then $[N]$ receive a color different from those in $\mathcal{B}_{[\emptyset, [N])}$.
    \item[$(2)$] Exactly two colors appear in $\mathcal{B}_{[\emptyset, [N])}$. The set $[N]$ receive a color, distinct from the two used in $\mathcal{B}_{[\emptyset, [N])}$.
\end{itemize}
When $(1)$ holds, we may fix such a set $A$ and find an element $a\in A$.
Thus $a\in X$ for each $X\in \mathcal{B}_{[A, [N]]}$ and $a\not\in X$ for each $X\in \mathcal{B}_{[\emptyset, [N]\setminus\{a\}]}$, and hence $\mathcal{B}_{[A, [N]]}\cap \mathcal{B}_{[\emptyset, [N]\setminus\{a\}]}=\emptyset$.
Then all sets in $\mathcal{B}_{[\emptyset, [N]\setminus\{a\}]}$ are monochromatically colored. 
Since 
$|[N]\setminus\{a\}|
=2e(\mathcal{P})\geq  \operatorname{Lu}_{N-1}(\mathcal{P})$, it follows that $\mathcal{B}_{[\emptyset, [N]\setminus\{a\}]}$ contains a monochromatic induced copy of $\mathcal{P}$, a contradiction. 

When $(2)$ holds. There use two colors for coloring of $\mathcal{B}_{N}\setminus\{N\}$. Since ${\rm R}_2(\mathcal{P})=2 e(\mathcal{P})=N-1$,  there is a monochromatic induced copy of $\mathcal{P}$, a contradiction. 
Therefore, $\mathrm{RR}(\vee_2:\mathcal{P})= 2e(\mathcal{P})+1$.
\end{proof}

\end{document}